\documentclass[a4,12pt]{amsart}%

\usepackage{amsfonts,amssymb,amsmath,amsthm}
\usepackage{algorithmicx, algorithm}
\usepackage{algcompatible}
\usepackage{color,ulem}
\usepackage{fullpage}
\usepackage{graphicx}%
\usepackage[active]{srcltx}
\usepackage{mathabx}
\usepackage{verbatim}
\usepackage{cancel}
\usepackage{mathtools}
\usepackage{tikz,pgfplots}

\usepackage{cite}
\usepackage{hyperref}
\hypersetup{
  linkcolor  = green!60!black,
  citecolor  = blue!60!black,
  urlcolor   = red!60!black,
  colorlinks = true,
}

\usepackage{caption}
\usepackage{subcaption}

\usepackage[official]{eurosym}

\providecommand{\U}[1]{\protect\rule{.1in}{.1in}}

\newtheorem{theorem}{Theorem}

\newtheorem{claim}[theorem]{Claim}

\newtheorem{lemma}[theorem]{Lemma}

\newtheorem{remark}[theorem]{Remark}

\def\d{\hbox{d}}

\usepackage{lipsum}
\makeatletter
\g@addto@macro{\endabstract}{\@setabstract}
\newcommand{\authorfootnotes}{\renewcommand\thefootnote{\@fnsymbol\c@footnote}}%
\makeatother

\usepackage{lipsum}
\usepackage{amsfonts}
\usepackage{graphicx}
\usepackage{epstopdf}
\ifpdf
  \DeclareGraphicsExtensions{.eps,.pdf,.png,.jpg}
\else
  \DeclareGraphicsExtensions{.eps}
\fi
\usepackage{cleveref}
\usepackage{tikz}
\usetikzlibrary{snakes}
\usepackage{amssymb}
\usepackage{framed}
\usepackage{subcaption}
\usepackage{pgfplots}
\usetikzlibrary{external}
\def\d{\hbox{d}}
\newcommand{\RR}{\mathbb{R}}
\newcommand{\sign}{\mathrm{sign}}
\def \Borel{\mathcal B}
\newcommand{\indic}{\mathbf{1}}
\makeatletter
\newcommand{\mylabel}[2]{#2\def\@currentlabel{#2}\label{#1}}
\makeatother

\begin{document}
\title{Penalization of non-smooth dynamical systems with noise : ergodicity and asymptotic formulae for threshold crossings probabilities}
\maketitle
\begin{center}
  \normalsize
  \authorfootnotes
 Mathieu Lauri{\`e}re\footnote{ORFE, Princeton University, Princeton, USA},
 \and
  Laurent Mertz\footnote{ECNU-NYU Institute of Mathematical Sciences,NYU Shanghai, Shanghai, China}
\end{center}
\begin{abstract}
 The purpose of this paper is to prove ergodicity and provide asymptotic formulae for probabilities of threshold crossing related to smooth approximations of three fundamental nonlinear mechanical models: (a) an elasto-plastic oscillator, (b) an oscillator with dry friction, (c) an oscillator constrained by an obstacle (one sided or two sided) and subject to impacts, all three in presence of white or colored noise. Relying on a groundbreaking result on density estimates for degenerate diffusions by Delarue and Menozzi (2010), we identify Lyapunov functions that satisfy appropriate conditions leading to ergodicity (invariant measure and Poisson equation) and a functional central limit theorem. These conditions appear in the very fundamental works of Down, Meyn and Tweedie (1995) and Glynn and Meyn (1996). From an applied mathematics perspective, an important consequence is the access to asymptotic formulae for quantities of interest in engineering and science. 
\end{abstract}



\section{Introduction}
\noindent Noise and vibration are fundamental features in an extremely wide range of industrial applications. In this context, mechanical structures will accumulate fatigue and then face a risk of failure. This is a major concern which has motivated a considerable effort in the engineering community e.g. see \cite{BV91,CL90,CAUGHEY1971209,grossmayer1981stochastic,LAZAROV2005251,MR0134505,MR1109057,R78,MR1408215RobertsReliability,MR2068684,V75,MR2676223} and reference therein. The challenge is to handle non-smooth stochastic dynamical systems. Strangely enough, it does not attract many experts in probability theory. While this is an important topic for research, mathematical references on the probabilistic analysis and numerical methods associated with such systems are still few in number. In this paper, we focus on smooth approximation of three natural phenomena involving interactions with boundaries, constraints, phase transitions or hysteresis: (a) an elasto-perfectly-plastic oscillator \cite{KS66plasticdefo}, (b) a dry friction one dimensional model \cite{Gennes2005}, (c) an oscillator constrained by an obstacle (one or two sided) and subject to impacts \cite{MR1647097}, all three in presence of noise.\\
 
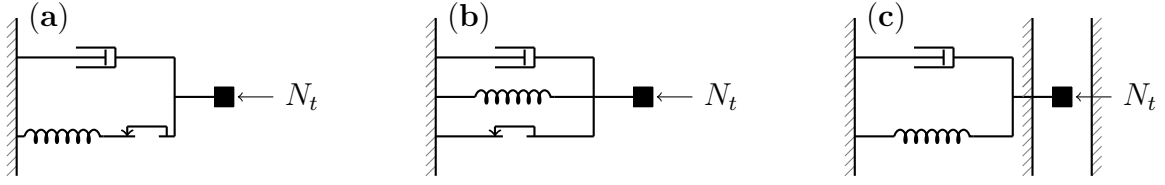
\begin{figure}[htbp]
  \centering
  \begin{tikzpicture}[scale=1.05]
\draw[thick] (-1.0,1.00) node[right] {$\bold{(a)}$};
\draw[color=gray] (-1.125,0.875) -- (-1,1);
\draw[color=gray] (-1.125,0.75) -- (-1,0.875);
\draw[color=gray] (-1.125,0.625) -- (-1,0.75);
\draw[color=gray] (-1.125,0.5) -- (-1,0.625);
\draw[color=gray] (-1.125,0.375) -- (-1,0.5);
\draw[color=gray] (-1.125,0.25) -- (-1,0.375);
\draw[color=gray] (-1.125,0.125) -- (-1,0.25);
\draw[color=gray] (-1.125,0) -- (-1,0.125);
\draw[color=gray] (-1.125,0.125) -- (-1,0.25);
\draw[color=gray] (-1.125,-0.125) -- (-1,0);
\draw[color=gray] (-1.125,-0.25) -- (-1,-0.125);
\draw[color=gray] (-1.125,-0.375) -- (-1,-0.25);
\draw[color=gray] (-1.125,-0.5) -- (-1,-0.375);
\draw[color=gray] (-1.125,-0.625) -- (-1,-0.5);
\draw[color=gray] (-1.125,-0.75) -- (-1,-0.625);
\draw[color=gray] (-1.125,-0.875) -- (-1,-0.75);
\draw[color=gray] (-1.125,-1) -- (-1,-0.875);
\draw[thick] (-1,-1) -- (-1,1) ;
\draw[thick] (-1,0.5) -- (0.125,0.5) ;
\draw[thick] (0.25,0.5) -- (1,0.5) ; 
\draw[thick] (-0.25,0.625) -- (0.25,0.625) ;
\draw[thick] (-0.25,0.375) -- (0.25,0.375) ;
\draw[thick] (0.25,0.375) -- (0.25,0.625) ;
\draw[thick] (0.125,0.4) -- (0.125,0.6) ;
\draw[thick] (-1,-.5) -- (-0.9,-.5) ;
\draw[snake=coil,segment length=4pt, thick] (-0.90,-.5) -- (0.1,-.5) ; 
\draw[thick] (0.1,-0.5) -- (0.5,-0.5) ;
\draw[thick] (0.4,-0.375) -- (0.9,-0.375) ;
\draw[thick,->] (0.4,-0.375) -- (0.4,-0.5);
\draw[thick] (0.9,-0.375) -- (0.9,-0.5);
\draw[thick] (.8,-0.5) -- (1,-0.5) ;
\draw[step=0.0125cm,color=black] (1.5,-0.125) grid (1.75,0.125);
\draw[thick] (1,-0.5) -- (1,0.5);
\draw[thick] (1,0) -- (1.5,0) ;
\draw[->]  (2.25,0) -- (1.80,0);
\draw[thick] (2.25,0) node[right] {$N_t$};
\end{tikzpicture}
\hspace{1cm}
\begin{tikzpicture}[scale=1.05]
\draw[thick] (-1.0,1.00) node[right] {$\bold{(b)}$};
\draw[color=gray] (-1.125,0.875) -- (-1,1);
\draw[color=gray] (-1.125,0.75) -- (-1,0.875);
\draw[color=gray] (-1.125,0.625) -- (-1,0.75);
\draw[color=gray] (-1.125,0.5) -- (-1,0.625);
\draw[color=gray] (-1.125,0.375) -- (-1,0.5);
\draw[color=gray] (-1.125,0.25) -- (-1,0.375);
\draw[color=gray] (-1.125,0.125) -- (-1,0.25);
\draw[color=gray] (-1.125,0) -- (-1,0.125);
\draw[color=gray] (-1.125,0.125) -- (-1,0.25);
\draw[color=gray] (-1.125,-0.125) -- (-1,0);
\draw[color=gray] (-1.125,-0.25) -- (-1,-0.125);
\draw[color=gray] (-1.125,-0.375) -- (-1,-0.25);
\draw[color=gray] (-1.125,-0.5) -- (-1,-0.375);
\draw[color=gray] (-1.125,-0.625) -- (-1,-0.5);
\draw[color=gray] (-1.125,-0.75) -- (-1,-0.625);
\draw[color=gray] (-1.125,-0.875) -- (-1,-0.75);
\draw[color=gray] (-1.125,-1) -- (-1,-0.875);
\draw[thick] (-1,-1) -- (-1,1) ;
\draw[thick] (-1,0.5) -- (0.125,0.5) ;
\draw[thick] (0.25,0.5) -- (1,0.5) ; 
\draw[thick] (-0.25,0.625) -- (0.25,0.625) ;
\draw[thick] (-0.25,0.375) -- (0.25,0.375) ;
\draw[thick] (0.25,0.375) -- (0.25,0.625) ;
\draw[thick] (0.125,0.4) -- (0.125,0.6) ;
\draw[thick] (-1,0) -- (-0.5,0) ;
\draw[snake=coil,segment length=4pt, thick] (-0.5,0) -- (0.5,0) ; 
\draw[thick] (0.5,0) -- (1,0) ;
\draw[thick] (-0.25,-0.375) -- (0.25,-0.375) ;
\draw[thick,->] (-.25,-0.375) -- (-.25,-0.5);
\draw[thick] (.25,-0.375) -- (.25,-0.5);
\draw[thick] (-1,-0.5) -- (-0.15,-0.5) ;
\draw[thick] (.15,-0.5) -- (1,-0.5) ;
\draw[step=0.0125cm,color=black] (1.5,-0.125) grid (1.75,0.125);
\draw[thick] (1,-0.5) -- (1,0.5);
\draw[thick] (1,0) -- (1.5,0) ;
\draw[->]  (2.25,0) -- (1.80,0);
\draw[thick] (2.25,0) node[right] {$N_t$};
\end{tikzpicture}
\hspace{1cm}
  \begin{tikzpicture}[scale=1.05]
\draw[thick] (-1.0,1.00) node[right] {$\bold{(c)}$};
\draw[color=gray] (-1.125,0.875) -- (-1,1);
\draw[color=gray] (-1.125,0.75) -- (-1,0.875);
\draw[color=gray] (-1.125,0.625) -- (-1,0.75);
\draw[color=gray] (-1.125,0.5) -- (-1,0.625);
\draw[color=gray] (-1.125,0.375) -- (-1,0.5);
\draw[color=gray] (-1.125,0.25) -- (-1,0.375);
\draw[color=gray] (-1.125,0.125) -- (-1,0.25);
\draw[color=gray] (-1.125,0) -- (-1,0.125);
\draw[color=gray] (-1.125,0.125) -- (-1,0.25);
\draw[color=gray] (-1.125,-0.125) -- (-1,0);
\draw[color=gray] (-1.125,-0.25) -- (-1,-0.125);
\draw[color=gray] (-1.125,-0.375) -- (-1,-0.25);
\draw[color=gray] (-1.125,-0.5) -- (-1,-0.375);
\draw[color=gray] (-1.125,-0.625) -- (-1,-0.5);
\draw[color=gray] (-1.125,-0.75) -- (-1,-0.625);
\draw[color=gray] (-1.125,-0.875) -- (-1,-0.75);
\draw[color=gray] (-1.125,-1) -- (-1,-0.875);
\draw[thick] (-1,-1) -- (-1,1) ;
\draw[thick] (-1,0.5) -- (0.125,0.5) ;
\draw[thick] (0.25,0.5) -- (1,0.5) ; 
\draw[thick] (-0.25,0.625) -- (0.25,0.625) ;
\draw[thick] (-0.25,0.375) -- (0.25,0.375) ;
\draw[thick] (0.25,0.375) -- (0.25,0.625) ;
\draw[thick] (0.125,0.4) -- (0.125,0.6) ;
\draw[thick] (-1,-0.5) -- (-0.5,-0.5) ;
\draw[snake=coil,segment length=4pt, thick] (-0.5,-0.5) -- (0.5,-0.5) ; 
\draw[thick] (0.5,-0.5) -- (1,-0.5) ;
\draw[step=0.0125cm,color=black] (1.5,-0.125) grid (1.75,0.125);
\draw[thick] (1,-0.5) -- (1,0.5);
\draw[thick] (1,0) -- (1.5,0) ;
\draw[->]  (2.25,0) -- (1.80,0);
\draw[thick] (2.25,0) node[right] {$N_t$};

\draw[color=gray] (1.125,0.875) -- (1.25,1);
\draw[color=gray] (1.125,0.75) -- (1.25,0.875);
\draw[color=gray] (1.125,0.625) -- (1.25,0.75);
\draw[color=gray] (1.125,0.5) -- (1.25,0.625);
\draw[color=gray] (1.125,0.375) -- (1.25,0.5);
\draw[color=gray] (1.125,0.25) -- (1.25,0.375);
\draw[color=gray] (1.125,0.125) -- (1.25,0.25);
\draw[color=gray] (1.125,0) -- (1.25,0.125);
\draw[color=gray] (1.125,0.125) -- (1.25,0.25);
\draw[color=gray] (1.125,-0.125) -- (1.25,0);
\draw[color=gray] (1.125,-0.25) -- (1.25,-0.125);
\draw[color=gray] (1.125,-0.375) -- (1.25,-0.25);
\draw[color=gray] (1.125,-0.5) -- (1.25,-0.375);
\draw[color=gray] (1.125,-0.625) -- (1.25,-0.5);
\draw[color=gray] (1.125,-0.75) -- (1.25,-0.625);
\draw[color=gray] (1.125,-0.875) -- (1.25,-0.75);
\draw[color=gray] (1.125,-1) -- (1.25,-0.875);
\draw[thick] (1.25,-1) -- (1.25,1) ;
\draw[color=gray] (2,0.875) -- (2.125,1);
\draw[color=gray] (2,0.75) -- (2.125,0.875);
\draw[color=gray] (2,0.625) -- (2.125,0.75);
\draw[color=gray] (2,0.5) -- (2.125,0.625);
\draw[color=gray] (2,0.375) -- (2.125,0.5);
\draw[color=gray] (2,0.25) -- (2.125,0.375);
\draw[color=gray] (2,0.125) -- (2.125,0.25);
\draw[color=gray] (2,0) -- (2.125,0.125);
\draw[color=gray] (2,0.125) -- (2.125,0.25);
\draw[color=gray] (2,-0.125) -- (2.125,0);
\draw[color=gray] (2,-0.25) -- (2.125,-0.125);
\draw[color=gray] (2,-0.375) -- (2.125,-0.25);
\draw[color=gray] (2,-0.5) -- (2.125,-0.375);
\draw[color=gray] (2,-0.625) -- (2.125,-0.5);
\draw[color=gray] (2,-0.75) -- (2.125,-0.625);
\draw[color=gray] (2,-0.875) -- (2.125,-0.75);
\draw[color=gray] (2,-1) -- (2.125,-0.875);
\draw[thick] (2,-1) -- (2,1) ;
\end{tikzpicture}
\caption{Rheological models: \textbf{(a)} an elasto-perfectly-plastic oscillator, 
\textbf{(b)} a dry friction one dimensional model and \textbf{(c)} an oscillator with an obstacle and impacts. A mass
(black box) is associated in series with elements which are themselves an association in parallel or in series of elementary rheological models. Each whole system is excited by a time-dependent random forcing $N_t$.}
\label{fig1}
\end{figure}
\noindent In section~\ref{sec-noise}, we present the type of noise that we consider. Then, in section~\ref{sec-svi}, we first write the models shown in Figure \ref{fig1} in terms of stochastic variational inequalities (SVIs); and then we put the focus on their smooth approximations by penalization. 
\subsection{White and colored noise}
\label{sec-noise}
\noindent In addition to a white noise forcing in the form $\dot W$ where $W$ is a real valued Wiener process, 
we will consider a type of stationary colored noise satisfying an \textit{overdamped Langevin dynamics} of the form
\begin{equation}
\label{noise1}
	\d \eta_t = -\boldsymbol{v}'(\eta_t) \d t + \sigma \d W_t 
\end{equation}
where $\boldsymbol{v}$ is a confining potential satisfying, for some constant $r>0$, 
\begin{equation}
\label{eq:cond-boldvprime}
	\boldsymbol{v}'(\eta) \eta \geq r |\eta|^2 .
\end{equation}
When $\sigma = \sqrt{2\beta^{-1}}$, $\beta >0$, it is well known that its invariant probability density function (PDF) on $\mathbb{R}$ is given by 
$$\rho(\eta) \triangleq Z_\beta^{-1} \exp \left (- \beta \boldsymbol{v}(\eta) \right ), 
\: Z_\beta \triangleq \int_{\RR} \exp \left (- \beta \boldsymbol{v}(\eta) \right ) \d \eta.$$
Typically, with $\boldsymbol{v}(\eta) = \frac{1}{2} \theta_{\boldsymbol{v}} \eta^2, \theta_{\boldsymbol{v}}>0$, it is clear that the so-called Ornstein-Uhlenbeck process belongs to this class of noise
and
$
\eta_t = \exp(-\theta_{\boldsymbol{v}} t) \eta_0 + \sigma \int_0^t \exp(-\theta_{\boldsymbol{v}}(t-s)) \d W_s.
$
In this case, its corresponding power spectrum is of the form 
$
P(\omega) = \frac{\sigma^2}{2 \theta_{\boldsymbol{v}} \pi} \frac{1}{\omega^2 + \theta_{\boldsymbol{v}}^2}, \: \omega \in \mathbb{R}.
$
In contrast with a white noise which has constant power spectrum, a colored noise is more interesting for applications since high frequencies are not present.  
%
The second class of colored noise belongs to the stochastic Hamiltonian systems of the form 
\begin{equation}
\label{noise2}
\d \zeta_t = B_1(\eta_t,\zeta_t) \d t + \sigma \d W_t,
\quad 
\d \eta_t = B_2(\eta_t,\zeta_t) \d t, \quad t > 0
\end{equation}
where 
\begin{equation}
\label{eq:defB1B2}
	B_1(\eta,\zeta) \triangleq - \partial_\eta H(\eta,\zeta) - F(\eta,\zeta) \partial_\zeta H(\eta,\zeta), 
	\quad B_2(\eta,\zeta) \triangleq \partial_\zeta H(\eta_t,\zeta_t).
\end{equation}
Here $H(\eta,\zeta)$ is a function called the Hamiltonian and $F$ is a function related to dissipation. Also, the existence and uniqueness of the invariant measure has been shown \cite{MR1924934} when the functions $H$ and $F$ satisfy a set of conditions (see Hypothesis 1.1, page 6 in \cite{MR1924934}). In particular, it is known that when $F \equiv C >0$ is a constant, $\sigma = \sqrt{2 C \beta^{-1}}$ ($\beta>0$) and the function $H(\eta,\zeta) \triangleq \frac{1}{2} \zeta^2 + \boldsymbol{v}(\eta)$, then $(\eta_t,\zeta_t)$ has a unique stationary measure state given by 
$$
\rho(\eta,\zeta) \triangleq C_\beta^{-1} \exp \left ( - \beta H(\eta,\zeta) \right ), \: C_\beta \triangleq \int_{\mathbb{R}^2}  \exp \left ( - \beta H(\eta,\zeta) \right ) \d \eta \d \zeta.
$$
\begin{remark}
\noindent 
In earthquake engineering, a realistic type of random forcing to represent \textit{seismic excitation} is the so-called Kanai-Tajimi model whose power spectral density is of the form
$
P(\omega) = \sigma \frac{k^2 + \gamma_0^2 \omega^2}{(k-\omega^2)^2 + \gamma_0^2 \omega}, \: \omega \in \mathbb{R}.
$
This type of noise belongs to the class of stochastic Hamiltonian systems mentioned right above with 
$
H(\eta,\zeta) = \frac{k}{2} \eta^2 + \frac{1}{2} \zeta^2
$
and $F \equiv \gamma_0$. Thus, $\eta$ of \eqref{noise2} is just the response of a white noise driven linear oscillator (with a stiffness $k$ and damping $\gamma_0$). Such a model reproduces dynamical properties of the ground \cite{LY87}.
\end{remark}

\subsection{Stochastic variational inequality formulation and penalization} 
\label{sec-svi}
\noindent We first recall a stochastic variational inequality (SVI) formulation for a class of non-smooth dynamical systems with noise.
It includes models shown in Figure \ref{fig1}. This type of mathematical structure has been identified as a solid framework in \cite{MR2426122,MR2225460,Bernardin04thesis}.
Then, we put the focus on smooth approximation of this type of system via the so-called Moreau-Yosida regularization \cite{MR0201952,MR1336382}.
In~\autoref{sec:appendixModelsSVI}, for the convenience of the reader, we explicitly write the SVIs and their penalized version for the elasto-plastic, obstacle, and friction problems.

\subsubsection{Stochastic variational inequality formulation}
\noindent Models shown in Figure \ref{fig1} can be described in terms of SVIs as follows: for a state variable $Z_t \triangleq (Y_t,X_t) \in \mathbb{R}^2$, where 
\begin{equation}
\label{nonsmooth1}
\tag{\ensuremath{\mathcal{SVI}-1}}
\forall (y,x) \in \mathbb{R}^2, \: \forall t >0, 
\: (b(Y_t,X_t) + N_t - \dot Y_t)(y - Y_t ) 
+ (Y_t - \dot X_t)(x - X_t ) + \varphi(Y_t,X_t) 
\leq \varphi(y,x), 
\end{equation}
or 
\begin{equation}
\label{nonsmooth2}
\tag{\ensuremath{\mathcal{SVI}-2}}
\forall x \in \mathbb{R}, 
\: \forall t >0, \: \dot X_t = Y_t \: \mbox{and} \: ( b(Y_t,X_t) + N_t - \dot Y_t )(x - X_t ) 
+ \varphi(X_t) 
\leq \varphi(x), 
\end{equation}
Here $b: \mathbb{R}^2 \to \mathbb{R}$, $b(y,x) \triangleq -U'(x) - C_b y$ where $U$ is a potential, $C_b >0$, 
$\varphi$ is a non-smooth convex real valued function on $\mathbb{R}^d (d= 1,2)$ and $N_t$ is a real valued stochastic forcing. 
See below Table~\ref{tabSVI-merged} for specifications of the function $\varphi$ in the different models, with the notation $K \triangleq [-1,1]$, and 
$\chi_K = 0$ on $K$ and $\infty$ on $\mathbb{R} - K$.
Throughout this paper, the forcing will be of three forms:  $N_t = \dot{W}$, $N_t = \eta_t$ of \eqref{noise1},  or $N_t = \eta_t$ of \eqref{noise2} and will be used as input in the three mechanical models. Note that the inequalities above must be satisfied for any $(y,x)$ or $x$ which are variational parameters. 

\subsubsection{Penalization of SVIs}
\label{sec-smooth}
\noindent 
The Moreau-Yosida regularization is a well-known approach to smooth any non-smooth convex function $\varphi$ as follows:
$$
\varphi_n(z) \triangleq \inf_{y \in \mathbb{R}^d} \left \{ \frac{n}{2} \| z-y \|^2 + \varphi(y) \right \}, \: \| y \| \triangleq \sqrt{y_1^2+ \dots + y_d^2}.
$$
Here, for every $n$, $\varphi_n(z)$ is differentiable with respect to $z$. For every $z \in \mathbb{R}^d, \lim_{n \to \infty} \varphi_n(z) = \varphi(z)$. 
Thus a regularized version of \eqref{nonsmooth1} or \eqref{nonsmooth2} consists in replacing $\varphi$ by $\varphi_n$. This procedure is also called penalization. 
Here, the Moreau-Yosida regularization of the functions $\chi_K$ and the absolute value $y \to |y|$ are respectively $\chi_n$ and $a_n$ defined as follows:
$$
\chi_n(x) \triangleq \frac{n}{2} |x-\textup{proj}_K(x)|^2
\: \mbox{and} \:
a_n(y) \triangleq 
\begin{cases} 
|y| - (2 n)^{-1} \: 					&\mbox{if} \: n |y| \geq 1\\ 
\dfrac{n}{2} y^2  \: &\mbox{if} \: n |y| < 1 
\end{cases}.
$$
In all cases (for the three mechanical models introduced above), this amounts to replace the SVI problem by a standard stochastic differential equation (SDE) with a nonlinear term depending on $n$ in the following sense:
\begin{equation}
\label{smooth}
\tag{\ensuremath{\mathcal P_n}}
\dot Z_t^n + f_{n}(Z_t^n) = B(Z_t^n) + \sigma_2 N_t, \: \mbox{ where } \: \sigma_2=(1,0)^T, \: B(y,x) \triangleq (b(y,x),y)^T. 
\end{equation}
Here $N_t = \dot W_t$ is a white noise or a colored noise as shown in \eqref{noise1} or \eqref{noise2}. Let us emphasize that a crucial point, for the mathematical perspective, is that the penalization allows us to have a unified treatment of the three mechanical models.\\ 
\begin{table}[h]
\centering
\caption{Specification of the the SVI framework and of the functions $\varphi, \varphi_n$ and $f_n$ for the three mechanical models of Figure~\ref{fig1}.}
\label{tabSVI-merged}
\begin{tabular}{|c|c|c|c|}
\hline
 &  Elasto-perfectly-plastic & Friction & Obstacle \\
\hline 
 SVI 
 & \eqref{nonsmooth1} 
 & \eqref{nonsmooth1} 
 & \eqref{nonsmooth2} \\
\hline 
 $\varphi$ 
 & $\varphi(y,x) = \chi_K(x)$ 
 & $\varphi(y,x) = |y|$
 & $\varphi(x) = \chi_K(x)$ \\
\hline 
 $\varphi_n$ 
 & $\varphi_n(y,x) = \chi_n(x)$ 
 & $\varphi_n(y,x) = a_n(y)$
 & $\varphi_n(x) = \chi_n(x)$\\
\hline 
 $f_n$ 
 & $f_n(y,x) = (0, \chi_n'(x))^T$ 
 & $f_n(y,x) = (a_n'(y),0)^T$
 & $f_n(y,x) = (\chi_n'(x),0)^T$\\
 \hline
\end{tabular}
\end{table}

\noindent \textbf{State variable extension and Markovian structure.} In order to remain in a Markovian framework \cite{MR933044}, the state variable $Z^n=(Y^n,X^n)$ is extended to $Z^n=(\eta,Y^n,X^n)$ for a noise \eqref{noise1} or $Z^n=(\zeta, \eta,Y^n,X^n)$ for a noise \eqref{noise2}. Therefore, we rewrite \eqref{smooth} in terms of a SDE of the form 
\begin{equation}
	\label{eq:dynZnFn}
	\d Z_t^n  = F_n(Z_t^n) \d t + \sigma \d W_t  
\end{equation}
with $F_n$ and $\sigma$ given in Table~\ref{tab-struct-b-sigma}. 

\small{\begin{table}[h!]
\centering
\caption{$F_n: \mathbb{R}^d \to \mathbb{R}^d$ and $\sigma \in \mathbb{R}^d$ for \eqref{eq:dynZnFn}. The functions $f_{n,x}$ and $f_{n,y}$ are defined, for each mechanical model, in Table~\ref{mectab}.}
\begin{tabular}{|c|c|c|}
\hline
 Type of noise & $F_n$ & $\sigma$ \\
\hline
    $
    \begin{matrix}
    & \textbf{white noise}\\
    & Z_t^n = (Y_t^n,X_t^n)^T
    \end{matrix}
    $
    & 
    $ 
    F_n \left ( \begin{matrix} y\\ x \end{matrix} \right ) 
    = \left ( 
    \begin{matrix} 
    -U'(x) - C_b y - f_{n,y}(y)\\ 
    y - f_{n,x}(x)
    \end{matrix} \right )
    $  
    &
    $\sigma =
     \begin{pmatrix}  1 \\ 0 \end{pmatrix} 
    $ 
\\ 
\hline
    $
    \begin{matrix}
    & \textbf{colored noise of type 1}\\
    & Z_t = (\eta_t, Y_t^n,X_t^n)^T
    \end{matrix}
    $
    & 
    $ 
    F_n \left ( \begin{matrix} \eta\\ y\\ x \end{matrix} \right ) 
    = \left ( 
    \begin{matrix} 
    -\boldsymbol{v}'(\eta)\\
    \eta -U'(x) - C_b y - f_{n,y}(y)\\ 
    y - f_{n,x}(x)
    \end{matrix} \right )
    $  
    &
    $\sigma =
     \begin{pmatrix}  1 \\ 0\\ 0 \end{pmatrix} 
    $ 
\\ 
\hline
    $
    \begin{matrix}
    & \textbf{colored noise of type 2}\\
    & Z_t^n = (\zeta_t,\eta_t, Y_t^n,X_t^n)^T
    \end{matrix}
    $
    & 
    $ 
    F_n \left ( \begin{matrix} \zeta \\ \eta \\ y\\ x \end{matrix} \right ) 
    = \left ( 
    \begin{matrix} 
    B_1(\eta,\zeta)\\
    B_2(\eta,\zeta)\\
    \eta -U'(x) - C_b y - f_{n,y}(y)\\ 
    y - f_{n,x}(x)
    \end{matrix} \right )
    $  
    &
    $\sigma =
     \begin{pmatrix}  1 \\ 0\\ 0\\ 0 \end{pmatrix} 
    $ 
\\ 
\hline
\end{tabular}
\label{tab-struct-b-sigma}
\end{table}
}

\begin{table}[h]
\centering
\caption{Explicit form of the functions $f_{n,y},f_{n,x}$ and potential for each mechanical model. For the obstacle problem, the trick is to move the penalization function directly into the potential.}
\label{mectab}
\begin{tabular}{|c|c|c|c|}
\hline
model & $f_{n,y}$ & $f_{n,x}$ & potential \\ \hline
elasto-plastic  & 0  &  $\chi_n'(x)$ & $U(x)$\\ \hline 
friction & $a_n'(y)$ & 0 & $U(x)$\\ \hline
obstacle & $0$ & $0$ & $U_n(x) \triangleq U(x) + \chi_n(x)$\\ \hline
\end{tabular}
\end{table}
\noindent In view of Table~\ref{mectab}, an important observation is that, in the three mechanical models (elasto-plastic, friction and obstacle)
the functions  $f_{n,x}$ and $f_{n,y}$ satisfy the following properties
\begin{itemize}
	\item[\mylabel{hyp:HPY}{(HPY)}] $|f_{n,y}| \leq 1$,
	\item[\mylabel{hyp:HPX}{(HPX)}] for all $x \in \RR$, $0 \leq \sign(x) f_{n,x}(x)  \leq n |x|$.
\end{itemize}
More general models can be treated in our framework, as long as \ref{hyp:HPY} and \ref{hyp:HPX} are satisfied.

The paper is organized as follows. 
Our main results are presented in~\cref{sec:mainresults} and are proved in~\cref{sec:proofLemma,sec:proofTheorems}. 
Applications and numerical results are provided in~\cref{sec:applications}.
Conclusions and future directions of research are proposed in~\cref{sec:conclusion}.
In~\cref{sec:appendixModelsSVI} we discuss the SVI framework for our mechanical models. 
For the sake of illustrations, examples of trajectories for each model are given in~\cref{sec:trajectories}.
Last, the numerical scheme used in~\cref{sec:applications} is explained in~\cref{sec:discretizationPDE}.

\section{Main results}
\label{sec:mainresults}
\subsection{Goal of the paper}
\label{sec-point}
Relying on \cite{MR2659772}, we prove the existence of a class of Lyapunov functions covering all the penalization version of the cases shown in Table~\ref{tab-struct-b-sigma} and satisfying the so-called \textit{Foster-Lyapunov} condition in the sense of \cite{MR1379163} and \cite{MR1404536}. 
This condition leads to ergodicity (existence and uniqueness of an invariant probability measure, rate of convergence for the semi-group, unique solution to the Poisson equation) and a functional central limit theorem. From an applied mathematics perspective, as an important consequence, we now have access to asymptotic formulae for the probabilities of threshold crossing for quantities of interest in engineering, physics and other fields.
\subsection{Standing assumptions and conditions}
We make the following assumptions.
\begin{framed}
\begin{itemize}
	\item \textbf{($\mathcal{A}_U$) - assumptions on $U$.} The potential $U : \mathbb{R} \to \mathbb{R}$ satisfies the following assumptions:
	\begin{itemize}
		\item[\mylabel{hyp:HU1b}{(HU1)}] $U$ is of class $\mathcal C^1$ and $U'$ is Lipschitz for some constant $\kappa$, that is, 
		$$
		\exists \kappa>0, \: \forall x,y \in \mathbb{R}, \: |U'(x)-U'(y)| \leq \kappa |x-y|.
		$$
		\item[\mylabel{hyp:HU2b}{(HU2)}] $\exists \beta_1>0, \lambda_1>0, \, U(x) \geq \lambda_1 x^2 - \beta_1$.
		\item[\mylabel{hyp:HU3b}{(HU3)}] $\exists \beta_2>0, \lambda_2>0, \, xU'(x) \geq \lambda_2 U(x) - \beta_2$.
	\end{itemize}
	Without loss of generality (up to an additive constant), we will also assume
	\begin{itemize}
		\item[\mylabel{hyp:HU4b}{(HU4)}]  $U \geq 0$.
	\end{itemize}
	\end{itemize}
	\end{framed}
	\noindent Note that, as a consequence of~\ref{hyp:HU2b} and~\ref{hyp:HU3b}, we also have
	\begin{equation}\label{hyp:HU5b} \tag{HU5}
		xU'(x) \geq \lambda_3 \Big[ U(x) + x^2\Big] - \beta_3, \: \mbox{where} \:
		\lambda_3 \triangleq \frac{\lambda_1 \lambda_2}{1+\lambda_1},
		\: \mbox{and} \: \beta_3 \triangleq \beta_2 + \frac{\lambda_2 \beta_1}{1+\lambda_1}.
	\end{equation}
	\begin{framed}
	\begin{itemize}
	%
	%
	\item \textbf{($\mathcal{A}_{H,F}$) - assumptions on $H, F$.}
	We assume that $H$ and $F$ are taken as in~\cite{MR1924934}, satisfying the same assumptions (see Hypothesis~1.1 of~\cite{MR1924934}) with $\nu = 2$. 
	More precisely, we assume $F$ and $H$ are of class $\mathcal C^\infty$ and there exist two constants 
	$\tilde \delta > 0, M >0$ and a function $R$ on $\RR^2$ such that for all $(\eta,\zeta) \in \RR^2$
	\begin{align}
		 H(\eta,\zeta) +  R(\eta,\zeta) + M \geq \tilde \delta \left(|\eta|^2 + |\zeta|^2\right)
		\label{eq:Talay117}
		\tag{T117}
		\\
		\left ( \frac{1}{2} \partial_{\zeta \zeta} + B_1 \partial_{\zeta} + B_2 \partial_{\eta} \right ) \left(  H(\eta,\zeta) + R(\eta,\zeta) \right )  
		\leq - \tilde \delta \left[ H(\eta,\zeta) + R(\eta,\zeta) \right] + M .
		\label{eq:Talay120}
		\tag{T120}
	\end{align}
\end{itemize}
We recall that $B_1$ and $B_2$ are defined by~\eqref{eq:defB1B2}.
We further assume that
\begin{align}
\zeta \mapsto \partial_{\zeta}B_2(\zeta,\eta) \: \mbox{is} \: \alpha-\mbox{H\"older continuous with a constant coefficient } \kappa.
\label{eq:1-ii-DM10-part1}\\
\exists l \in \RR, \quad \partial_{\zeta} B_2(\eta,\zeta) \geq l > 0, 
\quad \mbox{or} \quad
\partial_{\zeta}B_2(\zeta,\eta) \leq l < 0.		 
\label{eq:1-ii-DM10-part2}
\end{align}
\end{framed}
\noindent Inequalities~\eqref{eq:Talay117} and~\eqref{eq:Talay120} correspond to (1.17) and (1.20) of~\cite{MR1924934} respectively.


\subsection{Statement of the results}
\noindent 
From now on, unless otherwise specified $n$ is a fixed positive integer.
Let us consider $Z^n$ solving~\eqref{eq:dynZnFn}, with $\sigma = (1,0,\dots,0) \in \RR^d$ for simplicity. Define the transition semi-group $P_{n,t}$ of $Z^n$ as follows
\begin{equation}
\label{trans}
\forall \mathcal{O} \in \mathcal{B}(\mathbb{R}^d), \forall z \in \mathbb{R}^d,\: \forall t>0, \: P_{n,t}(z,\mathcal{O}) \triangleq \mathbb{P}(Z_t^n \in \mathcal{O} \vert Z_0^n = z)
\end{equation}
and the infinitesimal generator of $Z^n$ as
\begin{equation}
\label{inf-gen}
\forall \psi \in \mathcal C^2(\mathbb{R}^d), \quad A_n \psi(z) \triangleq 
\frac{1}{2} \partial_{z_1 z_1} \psi(z) + \sum_{i=1}^d F_{n,i}(z) \partial_{z_i} \psi(z), 
\quad z \in \mathbb{R}^d.  
\end{equation}
The notation $F_n$ is defined in Table~\ref{tab-struct-b-sigma} and $F_{n,i}$ is the $i^{\textup{th}}$ component of $F_n$.
The following Lemma shows that there exists a Lyapunov function satisfying two key properties~\cite{MR1379163}.
\begin{lemma}[Unbounded off petite set and Foster-Lyapunov drift conditions]
\label{lemma1}
Under the assumptions $\mathcal{A}_U$ and $\mathcal{A}_{H,F}$, there exists a function $V_n: \mathbb{R}^d \to [1,\infty)$ satisfying: 
\begin{itemize}
	\item[\mylabel{lemV-petitelevel}{(HV1)}] For all $r\geq 1$, the set $B_{V_n}(r) \triangleq \{z \in \mathbb{R}^d : {V_n}(z) \leq r\}$ is either empty or 
	satisfies: there exists a probability measure $\boldsymbol{a}$ on $\mathcal{B} ([0,\infty))$
	and a $\sigma$-finite measure $\nu$ on $\mathcal{B}(\mathbb{R}^d)$ such 
	that
	    \begin{equation}\label{eq:defPetite}
	    	\nu ( \cdot ) \leq \int_0^{\infty} P_{n,t}(z, \cdot) \boldsymbol{a} (\d t), \qquad \forall z \in B_{V_n}(r).
	    \end{equation}
	For $r$ fixed, this condition is referred to as $B_{V_n}(r)$ is petite and the whole condition is referred to as ${V_n}$ is \textit{unbounded off petite sets}.	
\item[\mylabel{lemV-driftcond}{(HV2)}] There exist two constants $\epsilon>0, \: C>0$ such that $\forall z \in \mathbb{R}^d, \: (A_n {V_n} + \epsilon {V_n}) (z) \leq C$.
\end{itemize}
\end{lemma}

\begin{remark}
The condition~\ref{lemV-petitelevel} can be formulated in other words : when $B_{V_n}(r)$ is not empty, there is a measure $\nu(.)$ on $\mathbb{R}^d$ and a random time $T_{\boldsymbol{a}}$ independent from $Z^n$ such that for every $B \in \mathcal{B}(\mathbb{R})$, the probability for $Z^n$ starting from an arbitrary point of $B_V(r)$ to be in $B$ at time $T_{\boldsymbol{a}}$ is larger than $\nu(B)$.
\end{remark}

\noindent These two conditions have been shown to entail ergodicity with a certain convergence rate (see Theorem 5.2 in~\cite{MR1379163}). 
To be more precise, let us recall the notion of $V$-uniform ergodicity. 
As presented in \cite{MR1379163}, we say that a transition semi-group $P_{t}$ is $V$-uniformly ergodic 
if there exist a probability measure $\pi$ and two constants $D >0$, $\rho \in (0,1)$ such that
	\begin{equation}
	\label{rate-convergence}
		|| P_{t}(z,.) - \pi ||_V \leq V(z) D \rho^t, \quad \forall \, t \geq 0, \: z \in \mathbb{R}^d,
	\end{equation}
where for every signed measure $\mu$ on $(\mathbb{R}^d, \Borel(\mathbb{R}^d))$, the $V$-norm of $\mu$ is
$$
||\mu||_V \triangleq \sup_{|g| \leq V} \left| \int_{\mathbb{R}^d} g(z)  \mu(\d z) \right|.
$$
Our first theorem states the existence and uniqueness of an invariant probability density for $Z^n$.
\begin{theorem}
\label{thm1}
The process $\{ Z_t^n, t \geq 0 \}$ is $V_n$-uniformly ergodic. Therefore, $Z^n$ has a unique invariant probability measure $\mu_n$ on $\mathbb{R}^d$, which admits a density $m_n$ with respect to Lebesgue measure on $\mathbb{R}^d$ and converges to it in the sense of~\eqref{rate-convergence} with an exponential rate. Moreover, $m_n$ is the unique solution of
\[
	\forall f \in \mathcal C^2(\mathbb{R}^d),\: \int_{\mathbb{R}^d} A_n f(z) m_n(z) \d z  = 0. 
\]
Therefore, $m_n$ is the unique probability density solution, in the sense of distributions, of 
$$
\frac{1}{2} \partial_{z_1,z_1} m_n = \nabla \cdot \left [ F_n m_n \right ], \quad \mbox{in} \quad \mathbb{R}^d \, .
$$ 
\end{theorem}
\noindent As mentioned above, the proof relies on Theorem 5.2 in~\cite{MR1379163}. Our second theorem states the existence and uniqueness of a solution to the Poisson equation for $Z^n$.
\begin{theorem}
\label{thm2}
Under~\ref{lemV-petitelevel} and~\ref{lemV-driftcond} given in Lemma \ref{lemma1}, the following properties hold 
\begin{enumerate}
	\item~$\{ Z_t^n, t \geq 0 \}$ is positive Harris recurrent, which in particular implies
            $$
            \forall f \: \mbox{s.t}, \: \int_{\mathbb{R}^d} |f(z)| m_n(z) \d z < \infty, 
            \quad \lim_{t \to \infty} \frac{1}{t} \int_0^t f(Z_s^n) \d s  =  \int_{\mathbb{R}^d} f(z) m_n(z) \d z, \: \mbox{a.s.}.
            $$
            and
            $$
            \int_{\mathbb{R}^d} V_n(z) m_n(z) \d z < \infty.
            $$
	\item ~For some constant $C>0$, the Poisson equation
            \begin{equation}
            \label{poisson}
            -A_n u = \bar{g} 
            \end{equation}
            admits a unique solution, which is denoted by $u^n_g$ in the sequel, and we have the bounds
            $$
            |u^n_g(z)| \leq C(V_n(z)+1), \: \mbox{ for a.e. } z .
            $$
            Here
            $$
            \bar{g} = g - \int_{\mathbb{R}^d} g(z) m_n(z) \d z, \quad \mbox{in} \quad \mathbb{R}^d. 
            $$
\end{enumerate}
\end{theorem}
The proof relies on Theorem~3.2 of~\cite{MR1404536}.

\noindent The third theorem is a functional central limit theorem for quantities of the form 
$$
	\Xi_{p,g}^n(t) \triangleq \frac{1}{\sqrt{p}} \int_0^{p t} \bar{g}(Z_s^n) \d s.
$$
\begin{theorem}
\label{thm3}
For any function $g$ such that $g^2 \leq V$,
there exists a constant $0 \leq \gamma_g^n < \infty$  such that for any initial distribution $\mu$, $\Xi_{p,g}^n \implies \gamma_g^n W$,
$\mathbb{P}_\mu$ - weakly  in $\mathbb{D}[0,1]$ as $p \to \infty$, where $W$ is a standard Wiener process.
Moreover, the constant $(\gamma_g^n)^2$ is characterized by combining the invariant measure of $Z^n$  and the solution to the Poisson equation~\eqref{poisson} as follows:
$$
	(\gamma_g^n)^2 = \int_{\mathbb{R}^d} \left | \partial_{z_1} u^n_g \right |^2(z) m_n(z) \d z.
$$
\end{theorem}
The proof relies on Theorem~4.4 of~\cite{MR1404536}.


\section{Proof of Lemma \ref{lemma1}} 
\label{sec:proofLemma}

\noindent Let us introduce the functions $V_{n,d}$, for $d \in \{2,3,4\}$, defined as
\begin{equation}
\label{Lyapunov-functions}
\begin{cases}
	V_{n,2}(y,x) &\triangleq \delta \left (\frac{y^2}{2} + U(x) \right ) + xy + C_V,
	\\ 
	V_{n,3}(\eta,y,x) &\triangleq \Gamma_1(\eta) + V_{n,2}(y,x),
	\\ 
	V_{n,4}(\zeta,\eta,y,x) &\triangleq \Gamma_2(\zeta,\eta) + V_{n,2}(y,x),
\end{cases}
\end{equation}
where 
\[
	\Gamma_1(\eta) \triangleq \frac{\xi}{2} \eta^2,
	\quad \Gamma_2(\zeta,\eta) \triangleq K \left ( H(\zeta,\eta) + R(\zeta,\eta) + M \right ).
\]
Here $\delta, C_V, \xi$ and $K$ are large enough constants. To be precise, we require that they satisfy the following inequalities:
\begin{framed}
\begin{equation}
\label{eq:LB-delta}
	\delta > \max \left ( \tfrac{1}{\sqrt{\lambda_1}}, \tfrac{2}{C_b} \left ( 2 + \tfrac{4(C_b+n)^2}{3 \lambda_3} \right ) \right ),
	\,
	C_V \geq 1 + \delta \beta_1,
	\,
	\xi > \tfrac{8}{r}\left[ \tfrac{\delta}{C_b} + \tfrac{2}{3\lambda_3}\right],
	\,
	K > \tfrac{4}{\tilde\delta^2}\left[\tfrac{\delta}{C_b}Ê+ \tfrac{2}{3 \lambda_3} \right].
\end{equation}
\end{framed}
We split the proof into two parts, the first one for~\ref{lemV-petitelevel} and the other one for~\ref{lemV-driftcond}.
\subsection{Proof of Lemma \ref{lemma1} :~\ref{lemV-petitelevel}}
We first show that, in each case, $V$ is unbounded off compact sets, and then that petite sets are compact. \textbf{First step:} Each function $V \in \{ V_{n,d} \}_{d = 2,3,4}$ is unbounded off compact sets in the sense that for all $r\geq 1$, the set $B_V(r) \triangleq \{z \in \mathbb{R}^d : V(z) \leq r\}$ is compact (possibly empty). Indeed, $V$ satisfies
	$V(z) \to \infty$ as $|z| \to \infty$. 
	This is checked by direct calculations as follows. For $V_{n,2}$, by assumption~\ref{hyp:HU2b} we have
	$$
		V_{n,2}(y,x) \geq \frac{y^2}{2} \left ( \delta - \frac{1}{\delta\lambda_1}  \right ) 
		+ \frac{\delta \lambda_1}{2}x^2 
		+ 1,
	$$
	and by~\eqref{eq:LB-delta} the coefficient of $y^2$ is strictly positive. 
	Also, we deduce readily that $V_{n,3}$ is unbounded off compact sets. 
	The same conclusion also holds for $V_{n,4}$ by using~\eqref{eq:Talay117}.
	\textbf{Second step:} We then check that compact sets are petite. To do so, we exploit density estimates provided by Theorem 1.1 of Delarue and Menozzi \cite{MR2659772}. It states (in a simplified form for our present problem~\eqref{eq:dynZnFn}) that 
        \begin{framed}
        \noindent Consider a chain of SDEs of the form,  
        \begin{equation}
        \label{eq:dynamics-DM10}
            d Z_t^1 = F_1(Z_t^1, \hdots, Z_t^d) d t + d W_t, 
            \quad      
            d Z_t^i = F_i(Z_t^{i-1}, \hdots, Z_t^d) d t, 
            \quad 2 \leq i \leq d,
        \end{equation}
with the initial condition $(Z_0^1, \hdots, Z_0^d) = z \in \mathbb{R}^d$ and the following conditions on $F_1, \hdots, F_d$:  
\begin{itemize}
    \item $F_1, \hdots, F_d$ are Lipschitz,
    \item for each $2 \leq i \leq n$, 
    \begin{enumerate}
        \item
        \begin{itemize}
        \item[(i)]
        $z_{i-1} \mapsto F_i(z_{i-1},z_i, \hdots, z_d)$ is continuously differentiable,
        \item[(ii)] 
        $z_{i-1} \mapsto \partial_{z_{i-1}} F_i(z_{i-1},z_i, \hdots, z_d) \equiv 1$ 
        except for $d=4$ and $i=2$ where it is $\alpha$-H\"older continuous with the coefficient $\kappa$.
        \end{itemize}
        \item (case $d = 4$)
        $\exists \ell, \forall z \in \mathbb{R}^{4}$
        $$
        \partial_{z_{1}} F_2(z) \geq \ell > 0 
        \quad 
        \mbox{or} 
        \quad 
        \partial_{z_{2}} F_2(z) \leq \ell < 0 \, .
        $$ 
    \end{enumerate}
\end{itemize}
Then at any time $t>0$ the solution of \eqref{eq:dynamics-DM10} admits a density $z' \in \mathbb{R}^d \mapsto p(t,z,z')$. Moreover, for any $T>0$, there exists a constant $C_T \geq 1$, depending on $T$ and $\kappa$ such that for any $0 < t \leq T$,
        \begin{equation}\label{eq:boundDelarueMenozzi}
                	C_T^{-1} t^{-\frac{d^2}{2}} \exp \left ( -C_T t \left | \mathbb{T}_t^{-1}(\theta_t(z) - z') \right |^2 \right )
                	\leq 
                	p(t,z,z') 
                	\leq 
                	C_T t^{-\frac{d^2}{2}} \exp \left ( -C_T^{-1} t \left | \mathbb{T}_t^{-1}(\theta_t(z) - z') \right |^2 \right ).
        \end{equation}
        Here $\mathbb{T}_t$ is a $d \times d$ diagonal matrix, called ``scale'' matrix, 
        where $(\mathbb{T}_t)_{i,i} = t^i, i = 1, \hdots, d$ and $(\mathbb{T}_t)_{i,j} = 0$ for $1 \leq i \neq j \leq d$.
        Also, $\theta_t(z) \triangleq (\theta_t^1, \hdots, \theta_t^d)$ is the solution of the deterministic ordinary differential equation (ODE)
         \begin{equation}
        \label{eq:dynamics-DM10-theta}
            d \theta_t^1 = F_1(\theta_t^1, \hdots, \theta_t^d) d t,      
            \quad d \theta_t^i = F_i(\theta_t^{i-1}, \hdots, \theta_t^d) d t, 
            \quad 2 \leq i \leq d
        \end{equation}        
        with the initial condition $\theta_0 = z$.
        \end{framed}
        \noindent Let $\mathcal K$ be a compact set of $\mathbb{R}^d$. Fix an arbitrary $z_0 \in \mathcal K$. From~\eqref{eq:boundDelarueMenozzi} we deduce that for any $\mathcal{O} \in \mathcal{B}(\mathbb{R}^d)$
        $$
                	\int_\mathcal{O} c_1(T) \exp \left ( - c_2(T) \sup_{t \in [0,T]} \left | \theta_t(z_0) - z' \right |^2 \right ) \d z'
                	\leq
                	\int_0^T \int_\mathcal{O} p(t,z_0,z')  \d z' \d t
        $$
        with $c_1(T) \triangleq C_T^{-1} T^{-\frac{d^2}{2}}$ and $c_2(T) \triangleq C_T T$. Hence~\eqref{eq:defPetite} is verified with $\boldsymbol{a}$ and $\nu$ defined by:
        \[
            	\boldsymbol{a}(\d t) \triangleq \frac{\indic_{[0,T]}(t)}{T}  \d t
            	\; \mbox{ and } \;
            	\nu(\mathcal{O}) 
            	\triangleq \frac{c_1(T)}{T}   \int_\mathcal{O} \exp \left ( - c_2(T) \sup_{z \in \mathcal K,t \in [0,T]} \left | \theta_t(z) - z' \right |^2 \right ) \d z'.
        \]
        Note that $\nu$ is well-defined since
        $(t,z) \mapsto \theta_t(z)$ and $z' \mapsto \sup_{z \in \mathcal K,t \in [0,T]} \left | \theta_t(z) - z' \right |$ are continuous functions.
\subsection{Proof of Lemma \ref{lemma1} :~\ref{lemV-driftcond}}
The proof is very similar in the three cases. For the sake of clarity we split the analysis into three parts, where we will use the same constant $\epsilon$, taken such that
\begin{equation}
\label{eq:def-epsilon-CV}
	0 < \epsilon < \min \left ( \frac{\lambda_3}{\delta}, C_b, r \right ).
\end{equation}

\subsubsection{Proof in the white noise case}
Let us denote by $A_{n,2}$ the infinitesimal generator of $Z_t^n = (Y_t^n,X_t^n)$ in the white noise case; see~\eqref{inf-gen} for the general expression of $A_n$. Thus, 
\[
A_{n,2}   
\triangleq
\frac{1}{2} \partial_{yy} - \left ( C_b y + U'(x) + f_{n,y}(y) \right ) \partial_y + (y - f_{n,x}(x)) \partial_x .
\]

\noindent We expand the left-hand side of the inequality in the statement of~\ref{lemV-driftcond} and obtain
\begin{align}
(A_{n,2} V_{n,2} + \epsilon V_{n,2})(y,x) 
 = & \left ( \frac{\delta}{2} + \epsilon C_V \right ) +y^2 \left ( 1 - \delta \left [ C_b - \frac{\epsilon}{2} \right ] \right ) 
 + \mathcal{S}_1(x)
 + \mathcal{S}_2(y,x),
	\label{expansion-V2}
\end{align}
with the notations
$$
	\mathcal{S}_1(x) \triangleq \delta \epsilon U(x) - U'(x) \left [ x + \delta f_{n,x}(x) \right ],
	\quad
	\mathcal{S}_2(y,x) \triangleq - y \left [ x (C_b -\epsilon) + f_{n,x}(x) \right ] - f_{n,y}(y) \left [ \delta y + x \right ].
$$

\noindent Bounds for $\mathcal{S}_1$ and $\mathcal{S}_2$ can be derived as shown in the two Claims below.
\begin{claim}[bound for $\mathcal{S}_1$]
For all $x \in \RR$,
\begin{equation}\label{eq:lyapu-bound-S1}
	 \mathcal{S}_1(x)
	 \leq
	 \Gamma - \tilde \Gamma x^2
\end{equation}
with $\tilde \Gamma \triangleq \frac{3 \lambda_3}{4}$ and $\Gamma \triangleq \beta_3 + \frac{(\delta \beta_4 n)^2}{\lambda_3}$, where $\beta_4 \triangleq \max \left ( \beta_3, \max_{|z|\leq 1} |U'(z)| \right )$.
\end{claim}

\noindent Note that $\beta_4 < \infty$ by~\ref{hyp:HU1b}.

\begin{proof}
	First, let us upper bound $-\delta f_{n,x}(x) U'(x)$. 	
	We have
	\begin{equation}\label{eq:UprimeGeq0}
		\sign(x)U'(x) + \beta_4 \geq 0.
	\end{equation}
	Indeed, from~\ref{hyp:HU4b} and~\eqref{hyp:HU5b}, 
	$$
		\sign(x) U'(x) + \beta_3 \geq 0, \: |x| \geq 1
	$$ 
	and from~\ref{hyp:HU1b}, 
	$$
		\sign(x) U'(x) +\max_{z \in [-1,1]} |U'(z)| \geq 0, \: |x| < 1.
	$$ 
	Using~\eqref{eq:UprimeGeq0} together with~\ref{hyp:HPX}, we obtain
	\begin{align}
		-\delta f_{n,x}(x) U'(x)
		&= \underbrace{-\delta \sign(x) f_{n,x}(x) \left[ \sign(x)U'(x) + \beta_4 \right]}_{ \leq 0} + \delta \beta_4 \sign(x) f_{n,x}(x)
		\notag
		\\
		&\leq \delta \beta_4 n |x|
		\notag
		\\
		&\leq \frac{(\delta \beta_4 n)^2}{\lambda_3} + \frac{\lambda_3}{4} x^2.
		\label{eq:genlyapu2-PXUprime-A}
	\end{align}
	
	\noindent Second, let us upper bound $- x U'(x) + \epsilon \delta U(x)$. Using the fact that $\epsilon < \frac{\lambda_3}{\delta}$ and~\eqref{hyp:HU5b},
\begin{equation}
- x U'(x) + \epsilon \delta U(x)  \leq \beta_3 - \lambda_3 x^2.
\label{eq:genlyapu2-PXUprime-B}
\end{equation}
	The conclusion holds by~\eqref{eq:genlyapu2-PXUprime-A} and~\eqref{eq:genlyapu2-PXUprime-B}.
\end{proof}
\begin{claim} [bound for $\mathcal{S}_2$] For all $x,y \in \RR$,
\begin{equation}\label{eq:lyapu-bound-S2}
	 \mathcal{S}_2(y,x)
	 \leq
	 \frac{(C_b+n)^2}{2\tilde\Gamma}y^2 + \frac{\tilde\Gamma}{2} x^2 + \delta |y| + |x|.
\end{equation}
\end{claim}
\begin{proof}
Using the fact that $0< \epsilon< C_b$ and~\ref{hyp:HPX} we obtain
\begin{equation*}
	| y \left [ x (C_b -\epsilon) + f_{n,x}(x) \right ] | 
	\leq 
	(n+C_b)|xy|
	\leq
	\frac{(C_b+n)^2}{2 \tilde \Gamma}y^2 + \frac{\tilde \Gamma}{2} x^2,
\end{equation*}
and using~\ref{hyp:HPY} we have
\begin{equation*}
	| f_{n,y}(y) (\delta y + x ) | \leq \delta |y| + |x|,
\end{equation*}
hence the result.
\end{proof}

\noindent We plug~\eqref{eq:lyapu-bound-S1}, and~\eqref{eq:lyapu-bound-S2} into~\eqref{expansion-V2}, and we use again the fact that $\delta > \frac{2}{C_b} \left ( 2 + \frac{4(C_b+n)^2}{3 \lambda_3} \right )$ to get
\begin{framed}
\begin{equation}
\label{An2V2epsV2-bound}
	\left (A_{n,2} V_{n,2} + \epsilon V_{n,2} \right )(y,x)  
	\leq  
	K_2 
	-  K_{2,y} \; y^2  
	- K_{2,x} \;x^2 
	+  \delta |y| + |x|,
\end{equation}
where
$$
	K_2 \triangleq C_V + \Gamma + \frac{\delta}{2}, 
	\quad 
	K_{2,y} \triangleq \frac{C_b\delta}{4},
	\quad
	K_{2,x} \triangleq \frac{\tilde \Gamma}{2}.
$$
\end{framed}
\noindent Since $K_{2,y}$ and $K_{2,x}$ are positive constants, there exists a large enough constant $C$ such that~\ref{lemV-driftcond} is satisfied.

\subsubsection{Proof in the case of colored noise of type~\eqref{noise1}}
\noindent
The proof follows a similar structure and uses the bound~\eqref{An2V2epsV2-bound} obtained in the white noise case.
We denote by $A_{n,3}$ the infinitesimal generator of $Z_t^n = (\eta_t,Y_t^n,X_t^n)$ in the colored noise of type~\eqref{noise1}, that is,
\[
A_{n,3} 
\triangleq
\frac{1}{2} \partial_{\eta \eta} - \boldsymbol{v}'(\eta) \psi_{\eta} 
+ \left ( \eta - \left [ C_b y + U'(x) + f_{n,y}(y) \right ] \right ) \partial_y
+ \left ( y - f_{n,x}(x) \right ) \partial_x.
\]
We have
\begin{equation}
\label{expansion-V3}
(A_{n,3} V_{n,3} + \epsilon V_{n,3})(\eta,y,x) 
 =  \mathcal{S}_3(\eta)
 + \eta [\delta y + x]
 + (A_{n,2} V_{n,2} + \epsilon V_{n,2})(y,x)  - \frac{\delta}{2},
\end{equation}
with the notation
$$
	\mathcal{S}_3(\eta) \triangleq \frac{1}{2} \Gamma_1''(\eta) - \boldsymbol{v}'(\eta) \Gamma_1'(\eta) 
 + \epsilon \Gamma_1(\eta).
$$
We show the following bound on $\mathcal{S}_3$.
\begin{claim}[bound for $\mathcal{S}_3$] 
For all $\eta \in \RR$
\begin{equation}
\label{eq:lyapu-bound-S3}
 	\mathcal{S}_3(\eta)
 	\leq \frac{\xi}{2} - \frac{\xi r}{2} \eta^2.
\end{equation}
\end{claim}
\begin{proof}
$
	\frac{1}{2} \Gamma_1''(\eta) - \boldsymbol{v}'(\eta) \Gamma_1'(\eta) 
	+ \epsilon \Gamma_1(\eta) =
	\frac{\xi}{2}
	- \boldsymbol{v}'(\eta) \xi \eta
	+ \epsilon \xi \frac{\eta^2}{2}
	\leq 
	\xi \left ( \frac{1}{2}
	- \eta^2 \left ( r - \frac{\epsilon}{2} \right ) \right )
$, where we used assumption~\eqref{eq:cond-boldvprime} on $\boldsymbol{v}$. The conclusion holds since $r > \epsilon$ by~\eqref{eq:def-epsilon-CV}.
\end{proof}
\noindent Using \eqref{eq:lyapu-bound-S3} and \eqref{An2V2epsV2-bound}, we get the following bound for \eqref{expansion-V3}:
\begin{framed}
\begin{equation}
\label{An3V3epsV3-bound}
(A_{n,3} V_{n,3} + \epsilon V_{n,3})(\eta,y,x) 
 \leq 
K_3
- K_{3,\eta}\;  \eta^2
- K_{3,y} \;  y^2
- K_{3,x}  \; x^2 
+ \delta |y| + |x|
\end{equation}
where
$$
	K_3 \triangleq C_V + \Gamma + \frac{\xi}{2}, 
	\quad 
	K_{3,\eta} \triangleq \frac{\xi r}{4}, 
	\quad 
	K_{3,y} \triangleq \frac{K_{2,y}}{2}, 
	\quad 
	K_{3,x} \triangleq \frac{K_{2,x}}{2}.
$$
\end{framed}
\noindent This yields~\ref{lemV-driftcond} for $d = 3$ (colored noise of type~\eqref{noise1}).

\subsubsection{Proof in the case of colored noise of type~\eqref{noise2}}
\noindent
Let us denote by $A_{n,4}$ the infinitesimal generator $Z_t^n = (\zeta_t, \eta_t,Y_t^n,X_t^n)$ in the colored noise of type~\eqref{noise2}, that is,
\[
	A_{n,4}  
	\triangleq \frac{1}{2} \partial_{\zeta \zeta} + B_1(\zeta,\eta) \partial_{\zeta} + B_2(\zeta,\eta) \partial_{\eta}  
	+ \left ( \eta - \left [ C_b y + U'(x) + f_{n,y}(y) \right ] \right ) \partial_y
	+ \left ( y - f_{n,x}(x) \right ) \partial_x.
\]
We have
\begin{align}
\label{expansion-V4}
	(A_{n,4} V_{n,4} + \epsilon V_{n,4})(\zeta,\eta,y,x) 
	 =& \mathcal{S}_4(\zeta, \eta) + \eta [\delta y + x] + (A_{n,2} V_{n,2} + \epsilon V_{n,2})(y,x)  - \frac{\delta}{2}.
\end{align}
with the notation
$$
	\mathcal{S}_4(\zeta, \eta)
	\triangleq
	\left ( \frac{1}{2} \partial_{\zeta,\zeta}\Gamma_2 + B_1 \partial_{\zeta}\Gamma_2 + B_2 \partial_{\eta}\Gamma_2 + \epsilon \Gamma_2 \right)(\zeta,\eta).
$$
Using~\eqref{eq:Talay117} and~\eqref{eq:Talay120}, one can check that the following bound on $\mathcal{S}_4$ holds.
\begin{claim}[bound for $\mathcal{S}_4$] 
For all $\zeta,\eta \in\RR$
\begin{equation}
\label{eq:lyapu-bound-S4}
	\mathcal{S}_4(\zeta, \eta) \leq M(1 + K \tilde\delta) - K \tilde \delta^2 \zeta^2 - K \tilde \delta^2 \eta^2.
\end{equation}
\end{claim}

\noindent Using \eqref{eq:lyapu-bound-S4} and \eqref{An2V2epsV2-bound}, we get the following bound for \eqref{expansion-V4}:
\begin{framed}
\begin{equation}
\label{An4V4epsV4-bound}
(A_{n,4} V_{n,4} + \epsilon V_{n,4})(\zeta, \eta,y,x) 
 \leq 
K_4
-  K_{4,\zeta} \; \zeta^2
- K_{4,\eta} \; \eta^2
- K_{4,y} \; y^2  
-  K_{4,x} \; x^2  
+ \delta |y| + |x|. 
\end{equation}
where
$$
	K_4 \triangleq C_V + \Gamma + M(1 + K \tilde\delta), 
 	\quad 
	K_{4,\zeta} \triangleq K \tilde\delta^2 , 
	\quad 
	K_{4,\eta} \triangleq \frac{K \tilde\delta^2}{2}, 
	\quad 
	K_{4,y} \triangleq \frac{K_{2,y}}{2}, 
	\quad 
	K_{4,x} \triangleq \frac{K_{2,x}}{2}.
$$
\end{framed}
\noindent This yields~\ref{lemV-driftcond} for $d = 4$ (colored noise of type~\eqref{noise2}).


\section{Proof of the Theorems}
\label{sec:proofTheorems}
\begin{proof}[Proof of Theorem~\ref{thm1}]
The existence of the unique invariant measure and the exponential convergence in the sense of $V_n$-uniform ergodicity of its semi-group are obtained by application of Theorem~5.1 (page 1679) of~\cite{MR1379163}. Indeed, our Lemma~\ref{lemma1} provides a Lyapunov function $V_n$ satisfying~\ref{lemV-petitelevel}--\ref{lemV-driftcond}, which implies (since $V_n$ is unbounded off petite sets) that there exist two constant  $b>0 $ and $c>0$, and a Borel petite set $\mathcal K$, such that  
    \begin{equation}
    \label{conditionTildeD}
     	\tag{\ensuremath{\tilde{\mathcal D}}} 
    	A_n V_n(z) \leq - c V_n(z) + b \indic_{\mathcal K}(z), \qquad \forall z \in \RR^d.
     \end{equation}
       To complete the proof of Theorem~\ref{thm1}, it remains to show that the limiting distribution has a density with respect to Lebesgue measure on $\RR^d$, denoted here $\lambda$. Let $B$ be a Borel subset of $\RR^d$ such that $\lambda(B)=0$. Since the process is $V_n$-uniformly ergodic, there exist a probability measure $\mu_n$ and two constants $D >0, \rho \in (0,1)$ such that, for all $t \geq 0$ and  $z \in \RR^d$,
                \begin{align*}
                	V_n(z) D \rho^t 
                	& \geq || P_{n,t}(z,.) - \mu_n ||_{V_n} = \sup_{|g| \leq V_n} \left| \int (P_{n,t}(z,.) - \mu_n)(\d z') g(z')\right|
                	\\
                	& \geq \left| P_{n,t}(z,B) - \mu_n(B) \right|
                	\\
                	& \geq \left| \lambda(B) - \mu_n(B) \right| - \left| P_{n,t}(z,B) - \lambda(B) \right|
                	\\
                	& \geq \left| \mu_n(B)\right|.
                \end{align*}
        For the second inequality, we take $g = \indic_{B}$, which is smaller than $V$ since $V \geq 1$. For the last inequality we use the fact that $P_{n,t}(z,.)$ is absolutely continuous with respect to $\lambda$. Letting $t \to \infty$ we obtain $\mu_n(B) = 0$. Hence $\mu_n$ too is absolutely continuous with respect to $\lambda$.
\end{proof}

\begin{proof}[Proof of Theorem~\ref{thm2}]
    The results are readily obtained by application of Theorem~3.2 (page 924) of~\cite{MR1404536}.
    Indeed, again our Lemma~\ref{lemma1} provides a Lyapunov function $V_n$ satisfying~\ref{lemV-petitelevel}--\ref{lemV-driftcond}, which implies 
     the Foster-Lyapunov drift condition of~\cite{MR1379163}, that is,  
     there exist a function $f:\RR^d \to [1,+\infty)$, a Borel petite set $\mathcal K$, and a constant $b < +\infty$ such that  
    \begin{equation}
    \label{conditionTildeDprime}
     	\tag{\ensuremath{\tilde{\mathcal D}'}} 
	    	A_n V_n(z) \leq -f(z) + b \indic_{\mathcal K}(z), \qquad \forall z \in \RR^d.
     \end{equation}
\end{proof}

\begin{proof}[Proof of Theorem~\ref{thm3}]
	The results are readily obtained by application of Theorem~4.4 (page 928) of~\cite{MR1404536}.
    Indeed, again our Lemma~\ref{lemma1} provides a Lyapunov function $V_n$ satisfying~\ref{lemV-petitelevel}--\ref{lemV-driftcond}, which implies that 
    $V_n$ satisfies condition (20) of~\cite{MR1404536}, namely: there exist two constants $b>0 $ and $c>0$, and a Borel petite set $\mathcal K$, such that~\eqref{conditionTildeD} holds.
\end{proof}


\section{Application to the penalization of the friction problem with white noise}
\label{sec:applications}

\noindent Our results can be used to study, for a broad class of functions $g$, quantities of the form
\begin{equation}
\label{eq:defWngDelta}
	W^n_g(b, T) 
	\triangleq 
	\mathbb{P} \left ( \max_{0 \leq t \leq T} |\Delta_g^n(t)| \geq b \right ), \: \mbox{where} \: \Delta_g^n(t) \triangleq \int_0^t g(Z^n_s) \d s .
\end{equation}
Indeed, we have
\begin{equation}
\label{equivalent}
	\lim_{p \to \infty} W^n_g(\sqrt{p}\, b, p T) = W_\star\left(\frac{b}{\gamma_g^n},T\right)
\end{equation}
where $W_\star(b,T)$ denotes the probability that the absolute value of a standard Wiener process crosses the threshold $b>0$ before time $T>0$, which is known to be given by
\begin{equation}
\label{eq:defFexplicit}
	W_\star(b, T) 
	\triangleq 
	1- \frac 4 \pi \sum_{k=0}^{+\infty} \frac {(-1)^k} {2k+1} \exp\left( -\frac{(2k+1)^2\pi^2 T}{8 b^2}  \right),
\end{equation}
and the coefficient $\gamma_g^n$ is obtained by combining the invariant probability measure of $Z^n$ with the solution to its Poisson equation, as shown in Theorem~\ref{thm3}. 
Regarding the non-smooth models, so far, such an asymptotic formula was available only in the elasto-perfectly-plastic setting with white noise~\cite{MR3733758}.
For the sake of illustration, we consider the example of the penalization of the friction problem shown in~\eqref{fpn} with $n$ large. For simplicity we consider the white noise case, and we take $C_b = 1$ and $U \equiv 0$ so the problem becomes one-dimensional (see~\cite{Bernardin04thesis} for more details.)

\subsection{Parabolic equation for $\gamma_g^n$}
From the Feynman-Kac formula, the quantity 
$$
	w_g^n(y,\tau) \triangleq \mathbb{E}_y \left [ \int_0^\tau g(Y^n_s) \d s \right ]
$$
is known to satisfy the linear parabolic equation
\begin{equation}
\label{eq:PDEw}
	\partial_\tau w_g^n(y,\tau) - \frac{1}{2} \partial_{yy}w_g^n(y,\tau) + (y + a_n'(y)) \partial_y w_g^n(y,\tau) = g(y,\tau), \: \forall (y,\tau) \in  \mathbb{R} \times [0,\infty),
\end{equation}
subject to the initial condition $w_g^n(y,0) = 0$.
Similarly, the quantity 
\begin{equation}
\label{eq:defvgyt}
	v_g^n(y,\tau) \triangleq \mathbb{E}_y \left [ \left (\int_0^\tau  \big[ g(Y^n_s) - \mathbb{E} [g(Y^n_s)] \big] \d s \right )^2 \right ]
\end{equation}
satisfies a similar equation but with $|\partial_y w_g^n|^2$ as a different right-hand side, namely,
\begin{equation}
\label{eq:PDEv}
	\partial_\tau v_g^n(y,\tau) - \frac{1}{2} \partial_{yy}v_g^n(y,\tau) + (y + a_n'(y)) \partial_y v_g^n(y,\tau) = \left | \partial_y w_g^n (y,\tau) \right |^2, \: \forall (y,\tau) \in  \mathbb{R} \times [0,\infty),
\end{equation}
subject to the initial condition $v_g^n(y,0) = 0$.


Then, from the ergodic property,  
\begin{equation}
\label{taularge-u}
\forall y \in \mathbb{R}, 
\quad 
	u_g^n(y,\tau) \sim \left (\int_{\mathbb{R}} g(y) m_n(y) \d y \right ), \quad \mbox{as} \quad  \tau \to \infty
\end{equation}
and
\begin{equation}
\label{taularge-v}
\forall y \in \mathbb{R}, 
\quad
	v_g^n(y,\tau) \sim (\gamma_g^n)^2 \tau, \quad \mbox{as} \quad \tau \to \infty.
\end{equation}
The key point for applications, is that we can use~\eqref{equivalent} to approximate probabilities of threshold crossings such as~\eqref{eq:defWngDelta} by using the explicit formula~\eqref{eq:defFexplicit} where an estimate of $\gamma_g^n$ is obtained via a solution to the PDE~\eqref{eq:PDEv}.

\subsection{Numerics} 
To illustrate how our approach can be used, we present here numerical results for the case where $g(y) = y$. 
A finite difference method is implemented in MATLAB to solve equations~\eqref{eq:PDEw} and~\eqref{eq:PDEv}, see Appendix~\ref{sec:discretizationPDE}. 
Figure~\ref{fig:fpn-cmp-PDE-MC} shows that the results from the PDE method and the Monte Carlo method are consistent. We used a logarithmic scale for time due to the exponential convergence rate towards the asymptotic value (roughly $1.36$ here).  For the Monte Carlo simulations we used $10^6$ samples.
Moreover, Figure~\ref{fig:fpn-cpm-formula} illustrates the convergence~\eqref{equivalent}: as $p$ increases, $W^n_g(\sqrt{p}b, pT)$ converges to the value given by the analytical formula ${W_\star}$. For the Monte Carlo simulations we used $10^5$ samples.
Table~\ref{tab:convergence-vng-n} shows the numerical the convergence of $v^n_g(0,T)$ as $n$ increases.

\begin{figure}[th]
\centering
\begin{subfigure}[t]{0.4\linewidth}
        \begin{tikzpicture}[scale=0.65]
        \begin{axis}[legend style={at={(1,0)},anchor=north west}, compat=1.3,
          xmin=0, xmax=10,ymin=0,ymax=0.14,
          xlabel= {$\tau$},
          ylabel= {$v^n_g(0,\tau)$},
          xmode=log,
          tick label style={/pgf/number format/fixed},
          legend style={at={(0,1)},anchor=north west},
          legend cell align=left]
         \addplot[solid,color=red,mark=none] table [x index=0, y index=1]{./FIG-FPn-1D_CMP-PDE-MC_DATA_data-PDE-FPn-v0t_Ly10_short.txt};
         \addplot[dashed,color=blue,mark=none] table [x index=1, y expr=\thisrowno{3}/\thisrowno{1}]{./FIG-FPn-1D_CMP-PDE-MC_DATA_data-nonsmooth-penalized-ex1d_n100_MC1000k_Nt100k_short.txt};
        \legend{PDE, MC}
        \end{axis}
        \end{tikzpicture}
        \caption{ Illustration of~\eqref{taularge-v}: the solid line represents $v^n_g(0,\tau)$, while the dashed line represents the term $(\gamma_g^n)^2 \tau$. The first one is computed by PDE method while the second one is computed using $10^6$ Monte Carlo samples.
          \label{fig:fpn-cmp-PDE-MC}}
  \end{subfigure}
  \quad
   \begin{subfigure}[t]{0.4\linewidth}
          \begin{tikzpicture}[scale=0.65]
        \begin{axis}[legend style={at={(1,0)},anchor=north west}, compat=1.3,
          xmin=0, xmax=20,ymin=0,ymax=1,
          xlabel= {$T$},
          ylabel= {$W^n_g$},
          tick label style={/pgf/number format/fixed},
          legend style={at={(1,0)},anchor=south east},
          legend cell align=left]
         \addplot[loosely dotted,color=black,mark=none] table [x index=0, y index=1]{./FIG-FPn-1D_CMP-FORMULA_DATA_data-nonsmooth-penalized-ex1d_probacross_n100_p1_MC100k_T20_Nt2k.txt};
         \addplot[dotted,color=brown,mark=none] table [x index=0, y index=1]{./FIG-FPn-1D_CMP-FORMULA_DATA_data-nonsmooth-penalized-ex1d_probacross_n100_p10_MC100k_T20_Nt2k.txt};
         \addplot[dashdotted,color=purple,mark=none] table [x index=0, y index=1]{./FIG-FPn-1D_CMP-FORMULA_DATA_data-nonsmooth-penalized-ex1d_probacross_n100_p100_MC100k_T20_Nt2k.txt};
        \addplot[dashed,color=blue,mark=none] table [x index=0, y index=1]{./FIG-FPn-1D_CMP-FORMULA_DATA_data-nonsmooth-penalized-ex1d_probacross_n100_p1000_MC100k_T20_Nt2k_s1.txt};
         \addplot[solid,color=red,mark=none] table [x index=0, y index=2]{./FIG-FPn-1D_CMP-FORMULA_DATA_data-nonsmooth-penalized-ex1d_probacross_n100_p1_MC100k_T20_Nt2k.txt};
          \legend{$p=1$, $p=10$, $p=100$, $p=1000$, ${W_\star}$}
        \end{axis}
        \end{tikzpicture}
        \caption{Illustration of~\eqref{equivalent}: the dashed and dotted lines represent $W^n(\sqrt{p}b, pT)$, for $4$ different values of $p$, while the solid line represent ${W_\star}(b/\gamma^n_g,T)$. Here $b=0.6, T=20$ and we used $10^5$ Monte Carlo samples.
        	\label{fig:fpn-cpm-formula}}
\end{subfigure}
\caption{\label{fig:fcpd-evolu3d} Numerical results for~\eqref{fpn} with white noise. Here, $g$ is the identity and $n = 100$.}
\end{figure}
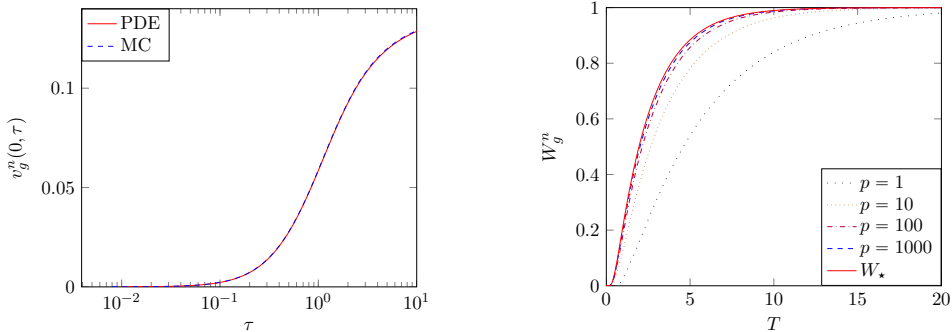

\begin{table}[h]
\centering
\caption{Convergence of $v^n_g(0,T)$ as $n$ increases, with $T=100$.}
\label{tab:convergence-vng-n}
\begin{tabular}{|c|c|c|c|c|c|c|}
\hline
 $n$ & $2$ & $5$ & $10$ & $50$ & $100$ & $1000$ \\
\hline
 $v^n_g(0,T)$ & $0.174466$ & $0.143272$ & $0.138434$ &  $0.136834$ & $0.136784$ & $0.136767$ \\
\hline
\end{tabular}
\end{table}

\section{Conclusion and related questions to be investigated}
\label{sec:conclusion}
\noindent 
In this work, we have proposed approximations for three non-smooth dynamical systems subjected to different kinds of noise. We have proved existence of Lyapunov functions, from which we have deduced ergodicity. This yields in particular a method to approximate probabilities of threshold crossing of quantities of interest for e.g. engineering and physics. The main ingredient of this work, Lemma \ref{lemma1}, crucially relies on the density estimates result of \cite{MR2659772}. From then on, a natural question is whether we can obtain such a result for the non-smooth problems. In fact, the aforementioned density estimates result cannot be employed and thus a corresponding version of Lemma \ref{lemma1} cannot be obtained using a similar approach. An other natural question to investigate is the behavior of our results as $n$ goes to $\infty$. We know \cite{MR3415395} that for the elasto-plastic problem, the solution of \eqref{eppn} converges, as $n$ goes to $\infty$, towards the solution of \eqref{svi_ep} in the sense of the following norm $\| X \| = \mathbb{E} \left ( \sup_{0 \leq t \leq T} |X(t)|^2 \right )$, as a consequence 
\begin{equation}
\label{limitWn}
\lim_{n \to \infty} W^n(b,T) = W(b,T).
\end{equation} 
However, it is not clear for us in how to prove (and in what sense) the convergence of the solutions of \eqref{fpn} and \eqref{opn} towards the solutions of \eqref{svi_fr} and \eqref{svi_op} respectively. Yet, Figures~\ref{fig:epp}, ~\ref{fig:op}, and~\ref{fig:fp} suggest that a result of the type of \eqref{limitWn} remains true. These issues remain to be investigated by other techniques and are beyond the scope of the present paper.

\appendix

\section{\eqref{nonsmooth1} and \eqref{nonsmooth2} and \eqref{smooth} for elasto-plastic, obstacle and friction problems}
\label{sec:appendixModelsSVI}
\subsection{Model definition of an elasto-plastic problem} \label{epo} 
The so-called elasto-plastic oscillator~\cite{KS66plasticdefo} can be explained as follows. 
The total deformation is described by $X_t^{\textup{Tot}} \in \mathbb{R}$
and its velocity by $Y_t \triangleq \dot{X}_t^{\textup{Tot}} \in \mathbb{R}$. For the elasto-perfectly-plastic oscillator (EPPO) model, the irreversible (plastic) deformation $\Delta$ and the reversible (elastic) deformation $X_t$ 
at time $t$ satisfy  
\begin{eqnarray*}
\d X_t & = \d X_t^{\textup{Tot}}, \quad \d \Delta_t & = 0, \quad \mbox{in elastic phase},\\
\d \Delta_t & = \d X_t^{\textup{Tot}}, \quad \d X_t & = 0, \quad \mbox{in plastic phase},
\end{eqnarray*}
while $X_t^{\textup{Tot}} = X_t + \Delta_t$. Typically, $|X_t|$ remains bounded by a given treshold $P_Y$ at all time $t$, plastic phase occurs when $ \vert X_t \vert = P_Y$ and elastic phase when $\vert X_t \vert < P_Y$. Here $P_Y$ is an elasto-plastic bound (known as the ``Plastic Yield'' in the engineering literature). Also, the permanent (plastic) deformation at time $t$ can then be
written as
\[
\Delta_t = \int_0^t  \mathbf{1}_{\{ |X_s| = P_Y \} } Y_s \hbox{d} s.
\] 
Due to the switching of regimes from an elastic phase to a plastic one, or vice versa, it is a non-smooth dynamical system.
We take $P_Y=1$ for simplicity. 
\subsection*{SVI framework}
\noindent It has been shown that the dynamics of a nonlinear
oscillator can be described mathematically by means of SVIs \cite{MR2426122}.
The dynamics is then described by the pair $(X_t,Y_t)$ that
satisfies
\begin{equation}
\label{svi_ep}
\tag{\ensuremath{\mathbf{EPP}}}
\begin{cases}
	\dot{Y}_t  =  N_t - \left ( U'(X_t) + C_b Y_t \right ),
	\\
	(\dot{X}_t-Y_t)(\zeta - X_t) \geq 0, \: \forall | \zeta | \leq 1, \: | X_t | \leq 1,
\end{cases}
\end{equation}
and appropriate initial conditions for $Y_0$ and $X_0$ must be prescribed. $U$ is a certain type of confining potential.
\subsection*{Penalized framework} 
The penalized version of~\eqref{svi_ep} is~:
\begin{equation}
\label{eppn}
\tag{\ensuremath{\mathbf{EPP}_n}}
\begin{cases}
	\dot{Y}^n_t = N_t - \left ( U'(X^n_t) + C_b Y^n_t \right ), 
	\\
	\dot{X}^n_t = Y^n_t  - n \big( X^n_t - \textup{proj}_{[-1,1]}(X^n_t)\big).
\end{cases}
\end{equation}
The penalization affects the elastic deformation $X^n$. See Figure~\ref{fig:epp}.  As $n$ tends to $\infty$, the penalization, acting on $\dot X$, enforces $X$ to remain between $-1$ and $+1$. 

\subsection{Model definition of an obstacle problem} 
\label{vio}
\noindent It is common in the engineering literature \cite{MR1647097} to formulate the dynamics of a stochastic obstacle oscillator in terms of
a stochastic process $X_t$, the oscillator displacement, which
evolves constrained by obstacles located at $|X|=P_O$ (position of the obstacle). For general obstacle problems, the non-smooth behavior in such models comes from the collisions in the sense that if at a time $t$, the state hits the obstacle with incoming velocity $\dot{X}_t$, it immediately bounces back with velocity $-\dot{X}_t$, that is, $\dot{X}_{t+} = - \dot{X}_{t-}$.
\subsection*{SVI framework} 
\noindent From a mathematical viewpoint, the dynamics of such a nonlinear
oscillator can be described in the framework of SVIs \cite{MR2225460}  as follows
\begin{equation}
\label{svi_op}
\tag{\ensuremath{\mathbf{OP}}}
\begin{cases}
	Y_t = \dot{X}_t,
	\\
	(\dot{Y}_t  - N_t + U'(X_t) + C_b Y_t)(\zeta - X_t)  \geq   0, \:  \forall  \vert \zeta \vert   \leq  P_O, \: \vert X_t \vert  \leq  P_O,
\end{cases}
\end{equation}
supplemented by the impact rule: $Y_{t+} = - Y_{t-}$. 
\subsection*{Penalized framework}
\noindent 	
We take $P_O = 1$ to simplify. 
Then, the penalized version of~\eqref{svi_op} is~:
\begin{equation}
\label{opn}
\tag{\ensuremath{\mathbf{OP}_n}}
\begin{cases}
	\dot{X}^n_t = Y^n_t, 
	\\
	\dot{Y}^n_t = N_t - \left ( U'(X^n_t) + C_b Y^n_t \right )  - n \left(X^n_t - \textup{proj}_{[-1,1]}(X^n_t) \right).
\end{cases}
\end{equation}
In contrast with ~\eqref{svi_ep}, note that here the penalization affects the velocity $Y$. 
See Figure~\ref{fig:op}.  As $n$ tends to $\infty$, the penalization, acting on $\dot Y$, enforces $X$ to remain between $-1$ and $+1$.

\subsection{Stick-slip friction}
A noise driven particle subject to friction can be expressed in terms of a stochastic process $X_t$ for the displacement and $Y_t = \dot{X}_t$ for its velocity. The non-smooth behavior comes from phase transitions between sticky and slippy phases. 
\subsection*{SVI framework} 
\noindent From a mathematical viewpoint, the dynamics of such a friction particle can be described in the framework of SVIs \cite{Bernardin04thesis}  as follows
\begin{equation}
\label{svi_fr}
\tag{\ensuremath{\mathbf{FP}}}
\begin{cases}
	\dot{X}_t = Y_t,
	\\
	(\dot{Y}_t  - N_t + U'(X_t) + C_b Y_t)(\varphi - Y_t)  \geq   P_F (|Y_t| - |\varphi|),
\:  \forall  \varphi \in \mathbb{R}.
\end{cases}
\end{equation}
Here $P_F$ is the amplitude of the friction force. Again we take $P_F =1$ for simplicity. 
\subsection*{Penalized framework}
The penalized version of~\eqref{svi_fr} is~:
\begin{equation}
\label{fpn}
\tag{\ensuremath{\mathbf{FP}_n}}
\begin{cases}
	\dot{X}^n_t = Y^n_t, 
	\\
	\dot{Y}^n_t =  N_t - \left ( U'(X^n_t) + C_b Y^n_t \right ) - a_n'(Y^n_t).
\end{cases}
\end{equation}
Here the penalization affects the velocity $Y^n$.
See Figure~\ref{fig:fp}. For each $n$, the penalization term, $a_n'(Y^n)$, acts as a term reverting $Y^n$ to $0$ (like a damping force). 
When is large, $Y^n$ remains $0$ most of the time. As $n$ tends to $\infty$, $a_n'(y)$ tends to $\sign(y) \in \{-1,+1\}$. 

%
\begin{remark}
\begin{enumerate}
	\item Note that the white noise driven \eqref{opn} model can be written as
$$
\d Y^n_t = - (U_n'(X^n_t) + C_b Y^n_t ) \d t + \d W_t, \quad \d X^n_t = Y^n_t \d t
$$
where $U_n(x) \triangleq U(x) + \frac{n}{2} | x - \textup{proj}_{[-1,1]}(x) |^2.$ Thus, using the notation $H_n(x,y) \triangleq \frac{1}{2} |y|^2 + U_n(x)$, the invariant probability density function of $(X_t,Y_t)$ is 
$$
m_n(x,y) = C_n^{-1} \exp \left ( - C_b H_n(x,y) \right ), \: C_n  \triangleq \int_{\mathbb{R}^2} \exp \left ( - \gamma H_n(x,y) \right ) \d x \d y.
$$ 
Here $H_n$ does not satisfies Hypothesis 1.1 of \cite{MR1924934} but the calculations for $m_n(x,y)$ are explicit. 
	\item The white noise driven \eqref{fpn} model has a Hamiltonian structure in the following sense
$$
\d Y^n_t = - \left (U'(X^n_t) + F_n(Y^n_t) Y^n_t \right )\d t + \d W_t, \quad \d X^n_t = Y^n_t \d t 
$$
where 
$$ 
F_n(y) \triangleq \begin{cases} C_b + \frac{1}{|y|}, & \: \mbox{if} \: |y| > \frac{1}{n}\\ C_b + n, & \: \mbox{if} \: |y| \leq \frac{1}{n} \end{cases}.
$$ 
In spite of the Hamiltonian structure, \cite{MR1924934} cannot be applied because $F_n$ does not satisfy Hypothesis 1.1 of~\cite{MR1924934} (it is not of class $\mathcal C^\infty$). As a consequence we have to resort to another technique which was previously employed for the white noise driven \eqref{eppn} model \cite{MR3415395} (with $U(x) \triangleq k \frac{x^2}{2}$). 
	\item We also observe that \eqref{eppn} model does not even have a Hamiltonian structure.
\end{enumerate}
\end{remark}

\section{Numerical simulation of trajectories} 
\label{sec:trajectories}
\noindent Below we provide a numerical illustration of the trajectory of $Z^n = (Y^n, X^n)$ in the white noise case for each problem (see Table~\ref{tab-struct-b-sigma}), for different values of $n$.
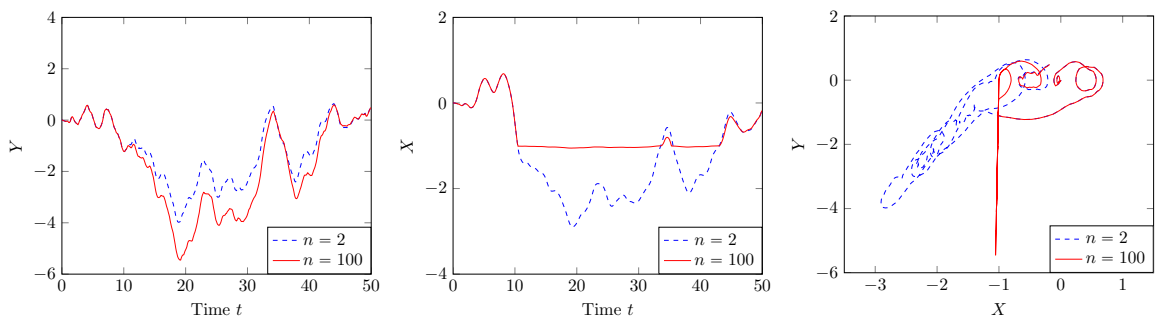
\begin{figure}[th]
\centering
\begin{tikzpicture}[scale=0.60]
\begin{axis}[legend style={at={(1,0)},anchor=south east}, compat=1.3,
  xmin=0, xmax=50,ymin=-6,ymax=4,
  xlabel= {Time $t$},
  ylabel= {$Y$},
  tick label style={/pgf/number format/fixed},
  legend cell align=left]
\addplot[dashed,color=blue,mark=none] table [x index=0, y index=2]{./FIGURE2_eppn1.txt};
 \addplot[solid,color=red,mark=none] table [x index=0, y index=2]{./FIGURE2_eppn3.txt};
 \legend{$n=2$,$n=100$}
\end{axis}
\end{tikzpicture}
\begin{tikzpicture}[scale=0.60]
\begin{axis}[legend style={at={(1,0)},anchor=south east}, compat=1.3,
  xmin=0, xmax=50,ymin=-4,ymax=2,
  xlabel= {Time $t$},
  ylabel= {$X$},
  tick label style={/pgf/number format/fixed},
  legend cell align=left]
 \addplot[dashed,color=blue,mark=none] table [x index=0, y index=3]{./FIGURE2_eppn1.txt};
 \addplot[solid,color=red,mark=none] table [x index=0, y index=3]{./FIGURE2_eppn3.txt};
 \legend{$n=2$,$n=100$}
\end{axis}
\end{tikzpicture}
\begin{tikzpicture}[scale=0.60]
\begin{axis}[legend style={at={(1,0)},anchor=south east}, compat=1.3,
  xmin=-3.5, xmax=1.5,ymin=-6,ymax=2,
  xlabel= {$X$},
  ylabel= {$Y$},
  tick label style={/pgf/number format/fixed},
  legend cell align=left]
 \addplot[dashed,color=blue,mark=none] table [x index=3, y index=2]{./FIGURE2_eppn1.txt};
 \addplot[solid,color=red,mark=none] table [x index=3, y index=2]{./FIGURE2_eppn3.txt};
 \legend{$n=2$,$n=100$}
\end{axis}
\end{tikzpicture}
\caption{Elasto-plastic problem: empirical convergence with respect to $n$. 
  \label{fig:epp}}
\end{figure}

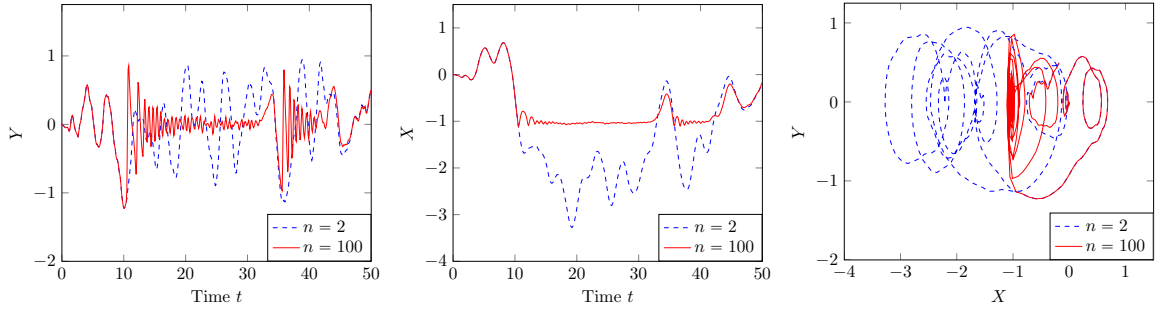
\begin{figure}[th]
\centering
\begin{tikzpicture}[scale=0.60]
\begin{axis}[legend style={at={(1,0)},anchor=south east}, compat=1.3,
  xmin=0, xmax=50,ymin=-2,ymax=1.75,
  xlabel= {Time $t$},
  ylabel= {$Y$},
  tick label style={/pgf/number format/fixed},
  legend cell align=left]
 \addplot[dashed,color=blue,mark=none] table [x index=0, y index=2]{./FIGURE2_opn1.txt};
 \addplot[solid,color=red,mark=none] table [x index=0, y index=2]{./FIGURE2_opn3.txt};
 \legend{$n=2$,$n=100$}
\end{axis}
\end{tikzpicture}
\begin{tikzpicture}[scale=0.60]
\begin{axis}[legend style={at={(1,0)},anchor=south east}, compat=1.3,
  xmin=0, xmax=50,ymin=-4,ymax=1.5,
  xlabel= {Time $t$},
  ylabel= {$X$},
  tick label style={/pgf/number format/fixed},
  legend cell align=left]
 \addplot[dashed,color=blue,mark=none] table [x index=0, y index=3]{./FIGURE2_opn1.txt};
 \addplot[solid,color=red,mark=none] table [x index=0, y index=3]{./FIGURE2_opn3.txt};
 \legend{$n=2$,$n=100$}
\end{axis}
\end{tikzpicture}
\begin{tikzpicture}[scale=0.60]
\begin{axis}[legend style={at={(1,0)},anchor=south east}, compat=1.3,
  xmin=-4, xmax=1.5,ymin=-2,ymax=1.25,
  xlabel= {$X$},
  ylabel= {$Y$},
  tick label style={/pgf/number format/fixed},
  legend cell align=left]
 \addplot[dashed,color=blue,mark=none] table [x index=3, y index=2]{./FIGURE2_opn1.txt};
 \addplot[solid,color=red,mark=none] table [x index=3, y index=2]{./FIGURE2_opn3.txt};
 \legend{$n=2$,$n=100$}
\end{axis}
\end{tikzpicture}
\caption{Obstacle problem: empirical convergence with respect to $n$. 
  \label{fig:op}}
\end{figure}

\begin{figure}[th]
\centering
\begin{tikzpicture}[scale=0.60]
\begin{axis}[legend style={at={(1,0)},anchor=south east}, compat=1.3,
  xmin=0, xmax=50,ymin=-1.75,ymax=1.75,
  xlabel= {Time $t$},
  ylabel= {$Y$},
  tick label style={/pgf/number format/fixed},
  legend cell align=left]
 \addplot[dashed,color=blue,mark=none] table [x index=0, y index=2]{./FIGURE2_fpn1.txt};
 \addplot[solid,color=red,mark=none] table [x index=0, y index=2]{./FIGURE2_fpn3.txt};
 \legend{$n=2$,$n=100$}
\end{axis}
\end{tikzpicture}
\begin{tikzpicture}[scale=0.60]
\begin{axis}[legend style={at={(1,0)},anchor=south east}, compat=1.3,
  xmin=0, xmax=50,ymin=-8,ymax=1,
  xlabel= {Time $t$},
  ylabel= {$X$},
  tick label style={/pgf/number format/fixed},
  legend cell align=left]
 \addplot[dashed,color=blue,mark=none] table [x index=0, y index=3]{./FIGURE2_fpn1.txt};
 \addplot[solid,color=red,mark=none] table [x index=0, y index=3]{./FIGURE2_fpn3.txt};
 \legend{$n=2$,$n=100$}
\end{axis}
\end{tikzpicture}
\begin{tikzpicture}[scale=0.60]
\begin{axis}[legend style={at={(1,0)},anchor=south east}, compat=1.3,
  xmin=-7, xmax=2,ymin=-1.75,ymax=1.75,
  xlabel= {$X$},
  ylabel= {$Y$},
  tick label style={/pgf/number format/fixed},
  legend cell align=left]
 \addplot[dashed,color=blue,mark=none] table [x index=3, y index=2]{./FIGURE2_fpn1.txt};
 \addplot[solid,color=red,mark=none] table [x index=3, y index=2]{./FIGURE2_fpn3.txt};
 \legend{$n=2$,$n=100$}
\end{axis}
\end{tikzpicture}
\caption{Friction problem: empirical convergence with respect to $n$. 
  \label{fig:fp}}
\end{figure}
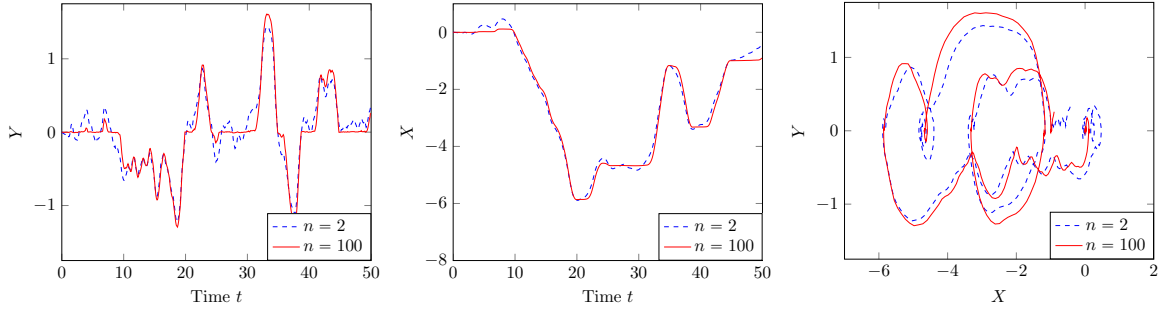

\section{Discretization of the partial differential equations \eqref{eq:PDEw}}
\label{sec:discretizationPDE}
\noindent To numerically approximate the solutions of~\eqref{eq:PDEw}, we use a finite difference scheme. We discretize the time with $t_k \triangleq k \Delta t$. 
We truncate the domain $\mathbb{R}$ into an interval $[-L,L]$ where $L$ is chosen sufficiently large that the probability of finding $Y_t^n$ outside $[-L,L]$ is negligible. We apply a homogeneous Neumann boundary condition, at $y = \pm L$. We consider a one-dimensional finite difference grid: $\mathcal{G} \triangleq \{ y_i = -L + (i-1) \Delta y, \: 1 \leq i \leq N \}$ with $\Delta y = \frac{2L}{N-1}$. 
The total number of nodes in $\mathcal{G}$ is $N$. The numerical approximations of $w_g^n(y_i,t_k)$ and $v_g^n(y_i,t_k)$ are denoted by $W_{i}^k$ and $V_{i}^k$ and the corresponding vectors are $\mathbf{W}$ and $\mathbf{V}$ (to alleviate the notations we drop the dependence on $g$). We also define $G_i \triangleq g(y_i)$ and denote by $\mathbf{G}$ the corresponding vector.
\subsection{Finite differences}
For every $2 \leq i \leq N-1$, at the point $y_i$ and a time $t_k$, $k \geq 1$, the discrete formulation of the parabolic equations \eqref{eq:PDEw} results in 
$$
\frac{W_i^k-W_i^{k-1}}{\Delta t} 
- \frac{1}{2} \left [ \frac{W_{i+1}^k - 2 W_{i}^k + W_{i-1}^k}{(\Delta y)^2} \right ] 
+ \underbrace{(y_i + a_n'(y_i))}_{b(y_i) \triangleq} \left [ \frac{W_{i+1}^k - W_{i-1}^k}{ 2 \Delta y} \right ]
= G_i
$$
The Neumann boundary conditions at the points $y_1$ and $y_N$ and the time $t_n$ $n \geq 1$ results in 
$$
\frac{W_N^k-W_{N-1}^k}{\Delta y} = 0 \quad \mbox{and} \quad \frac{W_2^k-W_1^k}{\Delta y} = 0.
$$
Similarly, the same discrete equation is satisfied 
for $V_{i}$ with $G_i$ replaced by $\left | (D W^k)_i \right |^2$
where
$$
(D W^k)_i \triangleq \frac{W_{i+1}^k - W_{i-1}^k}{ 2 \Delta y}.
$$
This results in the following linear systems to be solved for time $t_k$
$$
M_{\Delta t} \mathbf{U}^k = \mathbf{G}_{\Delta t}^k
\quad 
\mbox{and} 
\quad 
M_{\Delta t} \mathbf{V}^k = \mathbf{H}_{\Delta t}^k
$$
where, $M_{\Delta t}$ is a $N \times N$ matrix containing at most three non zeros entries per row. 
Precisely,
$$
(M_{\Delta t})_{1,1} = \frac{1}{\Delta y}, \quad (M_{\Delta t})_{1,2} = -\frac{1}{\Delta y},
$$
$$
(M_{\Delta t})_{i,i-1} = -\frac{\Delta t}{2} \left ( \frac{1}{(\Delta y)^2} + \frac{b(y_i)}{\Delta y} \right ),
\quad (M_{\Delta t})_{i,i} = 1 + \frac{\Delta t}{(\Delta y)^2},
$$
$$
(M_{\Delta t})_{i,i+1} = -\frac{\Delta t}{2} \left ( \frac{1}{(\Delta y)^2} - \frac{b(y_i)}{\Delta y} \right ),
$$
and
$$
(M_{\Delta t})_{N,N-1} = \frac{1}{\Delta y}, 
\quad (M_{\Delta t})_{N,N} = -\frac{1}{\Delta y}.
$$
Also,
$$
(G_{\Delta t}^k)_1 = (G_{\Delta t}^k)_N = 0 
\quad 
\mbox{and}
\quad 
(G_{\Delta t}^k)_i = \Delta t G_i + U_i^k, \quad 2 \leq i \leq N-1. 
$$
Similarly,
$$
(H_{\Delta t}^k)_1 = (H_{\Delta t}^k)_N = 0 
\quad 
\mbox{and}
\quad 
(H_{\Delta t}^k)_i = \Delta t (D U^k)_i^2 + V_i^k, \quad 2 \leq i \leq N-1. 
$$
We use a MATLAB implementation (available upon request). 


\section*{Acknowledgments}
\noindent LM would like to thank St\'ephane Menozzi and Alain Bensoussan for advises.
LM expresses his sincere gratitude to the Courant Institute
for being supported as Courant Instructor in 2014 and 2015, when this
work was initiated. LM and ML are supported by a faculty discretionary
fund from NYU Shanghai. LM is supported by the National Natural Science Foundation of
China, Research Fund for International Young Scientists under the
project \#1161101053 entitled ``Computational methods for non-smooth
dynamical systems excited by random forces'' and the Young Scientist
Program under the project \#11601335 entitled ``Stochastic Control
Method in Probabilistic Engineering Mechanics''.  LM also thanks Jonathan Wylie and the
Department of Mathematics of the City University of Hong Kong for the
hospitality and support.

\bibliographystyle{siamplain}
\bibliography{references-penalizationNSS}

\begin{thebibliography}{10}

\bibitem{MR1647097}
{\sc V.~I. Babitsky}, {\em Theory of vibro-impact systems and applications},
  Foundations of Engineering Mechanics, Springer-Verlag, Berlin, 1998,
  \url{https://doi.org/10.1007/978-3-540-69635-3}.
\newblock Translated from the Russian by N. Birkett and revised by the author.

\bibitem{MR2225460}
{\sc V.~Barbu and G.~Da~Prato}, {\em The stochastic obstacle problem for the
  harmonic oscillator with damping}, J. Funct. Anal., 235 (2006), pp.~430--448,
  \url{https://doi.org/10.1016/j.jfa.2005.11.004}.

\bibitem{MR2426122}
{\sc A.~Bensoussan and J.~Turi}, {\em Degenerate {D}irichlet problems related
  to the invariant measure of elasto-plastic oscillators}, Appl. Math. Optim.,
  58 (2008), pp.~1--27, \url{https://doi.org/10.1007/s00245-007-9027-4}.

\bibitem{Bernardin04thesis}
{\sc F.~Bernardin}, {\em Equations diff\'erentielles multivoques~: aspects
  th\'eoriques et num\'eriques - Applications.}, PhD thesis, Universit\'e
  Claude Bernard - Lyon I, 2004.

\bibitem{BV91}
{\sc B.~K. Bhartia and E.~H. Vanmarcke}, {\em Associate linear system approach
  to nonlinear random vibration}, Journal of Engineering Mechanics, 117 (1991),
  pp.~2407--2428,
  \url{https://doi.org/10.1061/(ASCE)0733-9399(1991)117:10(2407)}.

\bibitem{CL90}
{\sc G.~Q. Cai and Y.~K. Lin}, {\em On randomly excited hysteretic structures},
  Journal of Applied Mechanics, 57 (1990), pp.~442--448,
  \url{http://dx.doi.org/10.1115/1.2892009}.

\bibitem{CAUGHEY1971209}
{\sc T.~Caughey}, {\em Nonlinear theory of random vibrations}, vol.~11 of
  Advances in Applied Mechanics, Elsevier, 1971, pp.~209 -- 253,
  \url{https://doi.org/https://doi.org/10.1016/S0065-2156(08)70343-0},
  \url{http://www.sciencedirect.com/science/article/pii/S0065215608703430}.

\bibitem{MR2659772}
{\sc F.~Delarue and S.~Menozzi}, {\em Density estimates for a random noise
  propagating through a chain of differential equations}, J. Funct. Anal., 259
  (2010), pp.~1577--1630, \url{https://doi.org/10.1016/j.jfa.2010.05.002}.

\bibitem{MR1379163}
{\sc D.~Down, S.~P. Meyn, and R.~L. Tweedie}, {\em Exponential and uniform
  ergodicity of {M}arkov processes}, Ann. Probab., 23 (1995), pp.~1671--1691,
  \url{http://links.jstor.org/sici?sici=0091-1798(199510)23:4<1671:EAUEOM>2.0.CO;2-7&origin=MSN}.

\bibitem{MR3733758}
{\sc C.~Feau, M.~Lauri\`ere, and L.~Mertz}, {\em Asymptotic formulae for the
  risk of failure related to an elasto-plastic problem with noise}, Asymptot.
  Anal., 106 (2018), pp.~47--60.

\bibitem{MR0270403}
{\sc W.~Feller}, {\em An introduction to probability theory and its
  applications. {V}ol. {II}}, Second edition, John Wiley \& Sons, Inc., New
  York-London-Sydney, 1971.

\bibitem{Gennes2005}
{\sc P.~G.~d. Gennes}, {\em Brownian motion with dry friction}, Journal of
  Statistical Physics, 119 (2005), pp.~953--962,
  \url{https://doi.org/10.1007/s10955-005-4650-4},
  \url{https://doi.org/10.1007/s10955-005-4650-4}.

\bibitem{MR1404536}
{\sc P.~W. Glynn and S.~P. Meyn}, {\em A {L}iapounov bound for solutions of the
  {P}oisson equation}, Ann. Probab., 24 (1996), pp.~916--931,
  \url{https://doi.org/10.1214/aop/1039639370}.

\bibitem{grossmayer1981stochastic}
{\sc R.~L. Grossmayer}, {\em Stochastic analysis of elasto-plastic systems},
  Journal of the Engineering Mechanics Division, 107 (1981), pp.~97--116.

\bibitem{KS66plasticdefo}
{\sc D.~Karnopp and T.~D. Scharton}, {\em Plastic deformation in random
  vibration}, The Journal of the Acoustical Society of America, 39 (1966),
  pp.~1154--1161, \url{https://doi.org/10.1121/1.1910005},
  \url{https://doi.org/10.1121/1.1910005},
  \url{https://arxiv.org/abs/https://doi.org/10.1121/1.1910005}.

\bibitem{MR933044}
{\sc M.~M. K\l~osek Dygas, B.~J. Matkowsky, and Z.~Schuss}, {\em Colored noise
  in dynamical systems}, SIAM J. Appl. Math., 48 (1988), pp.~425--441,
  \url{https://doi.org/10.1137/0148023}.

\bibitem{MR3415395}
{\sc M.~Lauri\`ere and L.~Mertz}, {\em Penalization of a stochastic variational
  inequality modeling an elasto-plastic problem with noise}, in C{EMRACS}
  2013---modelling and simulation of complex systems: stochastic and
  deterministic approaches, vol.~48 of ESAIM Proc. Surveys, EDP Sci., Les Ulis,
  2015, pp.~226--247, \url{https://doi.org/10.1051/proc/201448010}.

\bibitem{LAZAROV2005251}
{\sc B.~Lazarov and O.~Ditlevsen}, {\em Slepian simulation of distributions of
  plastic displacements of earthquake excited shear frames with a large number
  of stories}, Probabilistic Engineering Mechanics, 20 (2005), pp.~251 -- 262,
  \url{https://doi.org/https://doi.org/10.1016/j.probengmech.2005.05.009},
  \url{http://www.sciencedirect.com/science/article/pii/S0266892005000135}.

\bibitem{LY87}
{\sc Y.~K. Lin and Y.~Yong}, {\em Evolutionary kanai-tajimi type earthquake
  models}, in Stochastic Approaches in Earthquake Engineering, Y.~K. Lin and
  R.~Minai, eds., Berlin, Heidelberg, 1987, Springer Berlin Heidelberg,
  pp.~174--203.

\bibitem{MR0134505}
{\sc R.~H. Lyon}, {\em On the vibration statistics of a randomly excited
  hard-spring oscillator. {II}}, J. Acoust. Soc. Amer., 33 (1961),
  pp.~1395--1403, \url{https://doi.org/10.1121/1.1908451}.

\bibitem{MR0201952}
{\sc J.-J. Moreau}, {\em Proximit\'e et dualit\'e dans un espace hilbertien},
  Bull. Soc. Math. France, 93 (1965), pp.~273--299,
  \url{http://www.numdam.org/item?id=BSMF_1965__93__273_0}.

\bibitem{MR1109057}
{\sc A.~Preumont}, {\em Vibrations al\'eatoires et analyse spectrale}, Presses
  Polytechniques et Universitaires Romandes, Lausanne, 1990.
\newblock With a preface by Andr\'e L. Jaumotte.

\bibitem{MR1408215RobertsReliability}
{\sc J.~Roberts}, {\em Reliability of randomly excited hysteretic systems}, in
  Mathematical Models for Structural Reliability Analysis, F.~Casciati and
  B.~Roberts, eds., CRC Press, Boca Raton, FL, 1996, pp.~139--194.

\bibitem{R78}
{\sc J.~B. Roberts}, {\em The response of an oscillator with bilinear
  hysteresis to stationary random excitation}, Journal of Applied Mechanics, 45
  (1978), pp.~923--928, \url{http://dx.doi.org/10.1115/1.3424442}.

\bibitem{MR2068684}
{\sc J.~B. Roberts and P.~D. Spanos}, {\em Random vibration and statistical
  linearization}, Dover Publications, Inc., Mineola, NY, 2003.
\newblock Revised reprint of the 1990 original [Wiley, Chichester; MR1076193].

\bibitem{MR1924934}
{\sc D.~Talay}, {\em Stochastic {H}amiltonian systems: exponential convergence
  to the invariant measure, and discretization by the implicit {E}uler scheme},
  Markov Process. Related Fields, 8 (2002), pp.~163--198.
\newblock Inhomogeneous random systems (Cergy-Pontoise, 2001).

\bibitem{MR2676223}
{\sc E.~Vanmarcke}, {\em Random fields}, World Scientific Publishing Co. Pte.
  Ltd., Hackensack, NJ, new~ed., 2010, \url{https://doi.org/10.1142/5807}.
\newblock Analysis and synthesis.

\bibitem{V75}
{\sc E.~H. Vanmarcke}, {\em On the distribution of the first-passage time for
  normal stationary random processes}, Journal of Applied Mechanics, 42 (1975),
  pp.~215--220, \url{http://dx.doi.org/10.1115/1.3423521}.

\bibitem{MR1336382}
{\sc K.~Yosida}, {\em Functional analysis}, Classics in Mathematics,
  Springer-Verlag, Berlin, 1995,
  \url{https://doi.org/10.1007/978-3-642-61859-8}.
\newblock Reprint of the sixth (1980) edition.

\end{thebibliography}

\end{document}